\documentclass[10pt,a4paper]{amsart}  

\usepackage{bbm,amsmath,amssymb,amsfonts,graphicx,color,subfig,booktabs,multirow,diagbox,mathtools}
\usepackage[hidelinks]{hyperref}
\usepackage[alphabetic]{amsrefs}    
\usepackage[english]{babel}

\newcommand{\R}{\mathbb{R}}
\newcommand{\N}{\mathbb{N}}
\newcommand{\Z}{\mathbb{Z}}
\newcommand{\C}{\mathbb{C}}

\newcommand{\Ccal}{\mathcal{C}}

\newcommand{\Mcal}{\mathcal{M}}
\newcommand{\Hcal}{\mathcal{H}}

\newcommand{\Vcal}{\mathcal{V}}

\def\spheresdp/{A}
\def\spacesdp/{B}

\newtheorem{defin}{Definition}[section]

\newtheorem{proposition}[defin]{Proposition}

\newtheorem{theorem}[defin]{Theorem}

\newtheorem{lemma}[defin]{Lemma}
\newtheorem{conjecture}[defin]{Conjecture}
\theoremstyle{definition}

\begin{document}

\title[Moment methods in energy minimization]{Moment methods in energy minimization: \\New bounds for Riesz minimal energy problems}
\author{\lowercase{\href{http://www.daviddelaat.nl}}{David de Laat}}

\thanks{Centrum Wiskunde \& Informatica (CWI), The Netherlands. Funded by Vidi grant
  639.032.917 and TOP grant 617.001.351 from the Netherlands Organization for Scientific
  Research (NWO).\\\indent A preliminary version of this paper appeared in the author's PhD thesis \cite{Laat2016}}



\date{15 August, 2019}

\begin{abstract}
We use moment techniques to construct a converging hierarchy of optimization problems to lower bound the ground state energy of interacting particle systems. We approximate (from below) the infinite dimensional optimization problems in this hierarchy by block diagonal semidefinite programs. For this we develop the necessary harmonic analysis for spaces consisting of subsets of another space, and we develop symmetric sum-of-squares techniques. We numerically compute the second step of our hierarchy for Riesz $s$-energy problems with five particles on the $2$-dimensional unit sphere, where the $s=1$ case is the Thomson problem. This yields new numerically sharp bounds (up to high precision) and suggests the second step of our hierarchy may be sharp throughout a phase transition and may be universally sharp for $5$-particles on the unit sphere. This is the first time a $4$-point bound has been computed for a problem in discrete geometry.
\end{abstract}

\maketitle

\section{Introduction}

\subsection{Background}\label{sec:background}

We consider the problem of finding the ground state energy $E$ of a system of $N$ interacting particles in a compact metric space $(V,d)$ with pair potential $F \colon \mathrm{diam}(V) \to \R$. The \emph{ground state energy} is given by the minimum of
\[
\sum_{1 \leq i < j \leq N} F(d(x_i, x_j))
\]
over all sets $\{x_1,\ldots,x_N\}$ of $N$ distinct points in $V$.  
An important example is the \emph{Thomson problem}, where we consider the unit sphere $V = S^2 \subseteq \R^3$ with the Euclidean metric $d(x,y) = \|x-y\|$ and Coulomb potential $F(c) = 1/c$. 

A simple proof that the configuration consisting of three equally spaced points on a great circle is optimal for the Thomson problem with $N=3$ was given in 1912 \cite{Foeppl1912}, but for $N > 3$ we seem to require more involved techniques. In 1992, Yudin \cite{Yudin1992} introduced a beautiful method, based on earlier work for spherical codes by Delsarte, Goethals, and Seidel \cite{DelsarteGoethalsSeidel1977}, which can be used to prove that the regular tetrahedron, octahedron, and icosahedron are optimal for the Thomson problems with $4$, $6$, and $12$ particles,  respectively \cites{Yudin1992,Andreev1996}. 

Yudin's bound (of degree $d$) for $N$ particles in $V=S^{n-1}$ with the Euclidean metric and pair potential $F$ is the following semi-infinite linear program in the nonnegative scalar variables $c_1,\ldots,c_d$: Maximize $(N^2-Nh(1))/2$ subject to the linear constraints 
\[
h(t) := 1 + \sum_{k=1}^d c_k C_k^{n/2-1}(t) \leq F(\sqrt{2-2t}) \quad \text{for} \quad t \in [-1,1).
\]
Here $C_k^{n/2-1}$ is the ultraspherical polynomial of degree $k$ in dimension $n$ normalized such that $\smash{C_k^{n/2-1}(1)}=1$.
These polynomials have the property that the matrix  
\[
\begin{pmatrix}C_k^{n/2-1}(x_i \cdot x_j)\end{pmatrix}_{i,j=1}^m
\]
is positive semidefinite for all $m\in\N$ and $x_1,\ldots,x_m \in S^{n-1}$. This property can be used to show that for each feasible function $h$ to the above semi-infinite linear program and each subset $C \subseteq S^{n-1}$ with $|C| = N$ we have
\[
N^2 = |C|^2 \leq \sum_{x,y \in C} h(x \cdot y) \leq N h(1) + \sum_{\substack{x,y \in C \\ x \neq y}} F(\sqrt{2-2 x\cdot y}),
\]
and since $\|x-y\|^2 = 2-2 x\cdot y$ for $x,y \in S^{n-1}$, this shows that for each degree $d$, Yudin's bound gives a lower bound on the ground state energy $E$.

For every $N \in \{2, 3, 4, 6, 12\}$ it has been shown that Yudin's bound is sharp for the Thomson problem, by which we mean that the optimal value of the bound meets the energy of an $N$-particle configuration. An optimal solution, consisting of the scalar variables $c_1,\ldots,c_d$ (for some $d$), then forms an optimality certificate, that is, an easy to check optimality proof, of the configuration.  
 
We can interpred Yudin's bound as the symmetrized version of the degree $d$ truncation of the dual of an infinite dimensional convex relaxation for the energy minimization problem. The goal of this paper is to find stronger convex relaxations and show how to symmetrize and compute duals of these relaxations.

Of particular interest are configurations that are optimal for a large class of pair potentials. We are for instance interested in configurations in $S^{n-1}$ that are optimal for each \emph{Riesz $s$-energy} potential $F(c) = \mathrm{sign}(s)/c^s$ with $s > 0$. In \cite{CohnKumar2007} Cohn and Kumar define a configuration in $S^{n-1}$ to be \emph{universally optimal} if it is optimal for all pair potentials $F$ of the form $F(c^2) = L(c)$, where $L$ is a completely monotonic function. Here a smooth function $L$ is said to be \emph{completely monotonic} if it is nonnegative, decreasing (that is, $F'$ is nonpositive), convex (that is, $F''$ nonnegative), and so on (the derivatives alternate in sign). Notice that the requirement that $L$ instead of $F$ is completely monotonic is a slightly stronger condition, but this is not important. It follows that a universally optimal configuration is optimal for each Riesz $s$-energy potential with $s > 0$.

In \cite{CohnKumar2007} Yudin's bound is used to show that each sharp configuration is universally optimal, and it is shown  that the $600$-cell, which is not a sharp configuration, is also universally optimal. Here a configuration $C$ on $S^{n-1}$ is said to be sharp if there are $m$ inner products between distinct points in $C$ and it is a $2m-1$-design; that is, for each polynomial in $n$ variables of degree at most $2m-1$ the average of the polynomial over $S^{n-1}$ equals the average over $C$.

In the derivation of Yudin's bound one considers conditions on pairs of particles, which is why this is called a \emph{$2$-point bound}. In \cite{CohnWoo2011} Cohn and Woo  derive $3$-point bounds for energy minimization based on earlier work by Bachoc and Vallentin \cite{BachocVallentin2008} for spherical codes and Schrijver \cite{Schrijver2005} for binary codes. They used this to prove universal optimality for the configuration consisting of the vertices of the rhombic dodecahedron in $\R\mathbb{P}^2$. In \cite{Musin2007} this is extended to $k$-point bounds, but here the sphere needs to be at least $k-1$ dimensional, which means this approach cannot be used to go beyond $3$-point bounds for energy minimization on the $2$-dimensional sphere $S^2$. Moreover, it is not clear how to use this to perform computations for $k > 3$.

\subsection{A converging hierarchy for energy minimization} 

In this paper we construct a hierarchy $E_1,E_2,\ldots,E_N$ of increasingly strong relaxations for energy minimization. Each $E_t$ is a convex minimization problem whose optimal value lower bounds the ground state energy $E$. To construct this hierarchy we use the moment methods developed in \cite{LaatVallentin2014} for packing problems in discrete geometry, which generalize techniques from the Lasserre hierarchy \cite{Lasserre2002a} in polynomial optimization to an infinite dimensional setting. 

In Section~\ref{sec:energy-lasserre} we show how the linear constraints in the optimization problems $E_t$ are derived. Here we derive enough constraints to ensure convergence of the hierarchy (in Section~\ref{sec:energy convergence to the optimal energy} we show $E_N = E$), but we keep the constraint set small enough to allow for a satisfying duality theory, which is crucial for performing computations. The $t$-th step $E_t$ in this hierarchy is a $\min\{2t,N\}$-point bound, and after symmetry reduction (see below), the first step $E_1$ becomes essentially the same as the relaxation whose dual gives Yudin's bound.

Each $E_t$ is an infinite dimensional minimization problems where we optimize over spaces of Radon measures. Naturally, this implies that in the dual maximization problem $E_t^*$ we optimize over spaces of continuous functions. We show how to approximate the dual problems $E_t^*$ by semidefinite programs that are block diagonalized into sufficiently small blocks so that it becomes possible to numerically compute the $4$-point bound $E_2$ for interesting problems. This leads to the best known lower bounds for these problems, and this demonstrates the computational applicability of the moment techniques developed in \cite{LaatVallentin2014}. 

In Section~\ref{sec:infinite dimensional optimization} we consider a class of infinite dimensional optimization problems that occur naturally when forming relaxations of problems in infinitely many binary variables. The relaxations $E_t$ fit into this framework, where each point in $V$ corresponds to a binary variable indicating whether this position is occupied by a particle. Similarly, the relaxations for geometric packing problems from  \cite{LaatVallentin2014} fit into this framework. To find good lower bounds we need to find good feasible solutions to the dual optimization problems, and for this we discuss duality and symmetry reduction for this class of problems. 

\subsection{Symmetry reduction}

We setup a computational framework, based around symmetry reduction, that will allow us to compute $E_2$. For this we use harmonic analysis, sum-of-squares characterizations, and semidefinite programming. These tools are also used for computing the $2$ and $3$-point bounds mentioned in Section~\ref{sec:background}, but for $t > 1$ our problems $E_t^*$ are of a rather different form.  

In the $2$ and $3$-point bounds for energy minimization on $V=S^2$, the dual variables are positive definite kernels $K \colon \smash{S^2} \times \smash{S^2} \to \R$, which are continuous functions for which the matrices $\smash{(K(x_i,x_j))_{i,j=1}^m}$ are positive semidefinite for all $m \in \N$ and $x_1,\ldots,x_m \in S^2$. For the $2$-point bounds these kernels are invariant under the orthogonal group $O(3)$ and for the $3$-point bounds under the stabilizer subgroup of a point in $S^2$. In our dual problems $E_t^*$, however, the variables are positive definite kernels $K \colon I_t \times I_t \to \R$, where $I_t$ is the set of independent sets of size at most $t$ in a certain graph $G$ with vertex set $V = S^2$. This graph $G$ inherits the symmetry of the problem, and we may assume $K$ to be invariant under this symmetry. This means that to compute $E_t^*$, we have to optimize over the cone of positive definite $O(3)$-invariant kernels $K \colon I_t \times I_t \to \R$.

In Appendix~\ref{sec:invariant posdef kernels} we show that for each compact metric space $V$ with a continuous action by a compact group $\Gamma$, the cone of positive definite $\Gamma$-invariant kernels is equal to the closure of the union of a sequence of certain simpler inner approximating cones $C_1, C_2,\ldots$. Each cone $C_d$ can be parametrized by a finite product of positive semidefinite matrix cones. By using harmonic analysis and symmetric tensor powers we show how such a sequence can be constructed explicitly, and we carry out this construction for the $V = S^2$ and $t = 2$ case (see Section~\ref{chap:energy:sec:explicit symmetry}). This gives an explicit formulation of a sequence of optimization problems $E_{2,d}^*$, each having finite dimensional variable space, whose optimal values converge to the optimal value of the $4$-point bound $E_2^* = E_2$ as $d \to \infty$.

For each fixed $d$, the problems $E_{t,d}^*$ have finite dimensional variable space, but still have infinitely many constraints. In Section~\ref{sec:reduction to sdps} we use invariant theory to write these problems as semidefinite programs with semialgebraic constraints. By a semialgebraic constraint we mean the requirement that a polynomial, whose coefficients depend linearly on the entries of the positive semidefinite matrix variable(s), is nonnegative on a basic closed semialgebraic set. We model these semialgebraic constraints as semidefinite constraints by using sum-of-squares characterizations from real algebraic geometry. In this way we obtain  a sequence of semidefinite programs $E_{t,d,\delta}^*$ with 
\[
E_{t,d,\delta}^* \leq E_{t,d}^* \leq E_t^* = E_t \leq E
\] 
and $E_{t,d,\delta}^* \to E_{t,d}^*$ as the sum-of-squares degree $\delta$ grows.

There exist efficient numerical algorithms for semidefinite programming (linear optimization over the intersection of an affine space with a product of positive semidefinite matrix cones), but for this it is important that the blocksizes are not too large. As described above, the problems $E_{t,d}^*$ are already block diagonalized through the use of harmonic analysis, which exploits the symmetry of the space $(V,d)$. In Section~\ref{sec:putinars theorem for invaraint} we use  additional symmetry, the interchangeability of the particles, to derive symmetries in the semialgebraic constraints. We then give a symmetrized version of Putinar's theorem, which allows us to exploit symmetries in semialgebraic constraints. This can lead to much smaller block sizes in the semidefinite programs, and, as we show by applying this to the problems $E_{2,d,\delta}^*$, this significantly reduces the computation time.

\subsection{Evidence for new sharp bounds} Although the $N$th relaxation $E_N$ is guaranteed to give the ground state energy $E$, the advantage of using a hierarchy is that $E_t$ can  be sharp for much smaller values of $t$. For example, Yudin's bound, which is essentially equal to the symmetry reduced version of $E_1^*$, is sharp for the Thomson problem for $N=2,3,4,6,12$. It would be very interesting if this pattern continues, and if $E_2$ would be sharp for several new values of $N$. In \cite{CohnWoo2011} the $3$-point bound is conjectured to be sharp for $N=8$, and since $E_2$ is a $4$-point it is also expected to be sharp for $N=8$. 

As a first step into investigating whether $E_2$ is sharp for new values of $N$ --- and to demonstrate that it is possible to compute the second step of our hierarchy --- we compute $E_{2,6,6}^*$ numerically (with high precision) for the $5$ particle case of the Thomson problem. The optimal value output of the high precision semidefinite programming solver consists of $28$ decimal digits, and we verify that these coincide with the first $28$ digits of the energy of the configuration given by the vertices of the triangular bipiramid. This is a strong indication that the bound is sharp. That is, the inequalities in 
\[
E_{2,6,6}^* \leq E_{2,6}^* \leq E_2^* = E_2 \leq E
\]
are all attained. This is the first time a $4$-point bound has been computed for a problem in discrete geometry. 

Since we use floating point computations, the solution given by the semidefinite programming solver is near feasible and near optimal, and hence does not immediately give an optimality proof. However, the $2$ and $3$ point bounds discussed above have been successfully used to obtain proofs by rounding the solver output. In \cite{CohnWoo2011}, for example, a $3$-point semidefinite programming bound was used to prove optimality of the rhombic dodecahedron. Here the floating point solver output was rounded to a solution (consisting of algebraic numbers) that lies on the optimal face, and a small computer program was used to verify (in exact arithmetic) that the solution is indeed feasible for the semidefinite program. There are no principal objections to using the same approach for the bounds $E_{2,d,\delta}^*$. The only difference is that in our case not only the solver uses floating point arithmetic, but also the solver input consists of high precision floating point numbers, since we use numerical linear algebra to generate the input for the semidefinite programming solver. One possible approach to generate the solver input exactly is to use Gr\"obner bases instead of numerical linear algebra. We  do not do in this paper, as the goal here is to formulate the new bounds $E_t$, show that and how the $4$-point bound $E_2$ can be computed numerically, and to identify new cases of sharp bounds. 

The case of $5$ particles on $S^2$ is particularly interesting because it is one of the simplest mathematical models admitting a phase transition, where by a \emph{phase transition} we mean that a slight change of the pair potential $F$ results in a discontinuous jump from one global optimum to another. Consider for instance the \emph{Riesz $s$-energy} potential $F(c) = \mathrm{sign}(s)/c^s$.
In \cite{MelnikKnopSmith1977} it is conjectured that the configuration consisting of the vertices of the triangular bipyramid is optimal for $0 < s \leq s^* := 15.048\dotso$, and the configuration consisting of the vertices of a square pyramid (where the latitude of the base depends on the value of $s$) is optimal for $s \geq s^*$. There is much recent progress on this problem \cites{Schwartz2016, Schwartz2016b}, and in the book \cite{Schwartz2016b} Schwartz proves this conjecture for all $s$ between $0$ and a number slightly bigger than $s^*$. 

The approach of Schwartz consists of three main steps. The first step is similar in some sense to the reduction in \cite{CohnWoo2011} of proving universal optimality of a configuration to the problem of proving optimality with respect to a finite list of potential functions. This starts with the result by Tumanov \cite{Tumanov2013} that if a configuration of $5$ points in $S^2$ is optimal for the pair potentials $G_k(c) = (4-c^2)^k$ with $k=2, 3, 5$, then it is optimal for all Riesz $s$-energy potentials with $s \in (-2,0) \cup (0, 2]$. Schwartz improves on this by giving a finite list of potentials for the larger interval $(-2,0) \cup (0, 13)$. The second main step consists in proving optimality for each of these finitely many pair potentials. This is done through derivative bounds and a very efficient discretization of the space of all $5$ point configurations, which depends on the specific geometry arising from considering $5$ points on $S^2$. In the third step it is shown that if a $5$ point configuration is not the triangular bipiramid and is optimal for a Riesz potential with $13 < s \leq s_0$, where $s_0$ is an explicit number slightly bigger than $s^*$, then it must lie in a certain two dimensional space of configurations. From this it is then derived that the configuration must be a square bipiramid.

Instead of searching through all configurations, our approach is to  find sharp solutions to dual optimization problems. As mentioned above, these dual solutions can be used to construct easy to verify optimality proofs (optimality certificates). Our method does not use problem specific geometric arguments, and instead of discretizing the space of configurations we use a truncated Fourier series in the dual. The computational framework that we setup in this paper can also be applied to energy problems with more than $5$ particles, and to packing problems, where searching through the feasible configurations would not be feasible. That $E_1$ and $E_2$ can be sharp for configurations with $N > 2$ and $N > 4$ particles, respectively, suggests a deeper connection between optimal geometric configurations and positive semidefinite constraints on these configurations. As explained below, this approach can also give more insight into why certain configurations are (universally) optimal and why proving optimality for some instances is harder than others. The above relates primarily to the second main step in Schwartz's approach. As mentioned below we conjecture that the bound $E_2^*$ is sharp for all Riesz $s$ potentials, which relates to the third main step in Schwartz's approach.

In addition to the $s=1$ case (the Thomson problem), we compute $E_2$ for Riesz $s$-energy potentials with $s = 2,\ldots,7$, where $\smash{E_{2,6,6}^*}$ is numerically sharp for $s = 1,\ldots,5$ and $\smash{E_{2,6,8}^*}$ is numerically sharp for $s = 6,7$. Here we again verify that the first $28$ digits of the solver output agree with the energy of the corresponding configuration. Since we need to increase the parameter $\delta$ as $s$ gets larger, we have not been able to compute other sharp bounds. 
Based on the above evidence we have the following conjecture:
\begin{conjecture}
\label{conjecture:riesz}
The bound $E_2$ is sharp for the minimal Riesz $s$-energy of $5$ particles on $S^2$ for $s = 1,\ldots,7$.
\end{conjecture}
In \cite{CohnWoo2011}*{Conjecture 15} it is conjectured that if there exists a completely monotonic potential function (in the squared Euclidean distance) for which a $k$-point bound is sharp for $N$ particles on $S^{n-1}$, and if this function is not a polynomial, then this $k$-point bound is universally sharp for $N$ particles on $S^{n-1}$. Here by \emph{universally sharp} we mean that the bound gives the ground state energy for all completely monotonic pair potentials in the squared chordal distance. If this conjecture is true for our version $E_2$ of $4$-point bounds for energy minimization, then the validity of Conjecture~\ref{conjecture:riesz} implies that $E_2$ is sharp for all positive $s$, and is thus sharp throughout a phase transition. In fact, then the following conjecture holds:
\begin{conjecture}
The bound $E_2$ is universally sharp for $5$ particles on $S^2$.
\end{conjecture}

Given a particle configuration on a sphere (or, more generally, on a $2$-point homogeneous space), there exists a beautiful sufficient criterion, based on geometric design theory, that says Yudin's bound is sharp for this configuration \cites{CohnKumar2007, Levenshtein1992}. No such criterion is known for $k$-point bounds with $k > 2$. Together with the newly found numerically sharp instances of $E_2$ and possibly other sharp instances of $E_2$ still to be discovered (see Section~\ref{sec:enegy:computations}), the new approach in this paper to formulating $k$-point bounds for energy minimization may help in formulating such a sufficient criterion. This may help in getting a better understanding of the geometric ideas behind optimality and universal optimality of particle configurations.

\section{A hierarchy of relaxations for energy minimizaton} 
\label{sec:derivation of the hierarchy via relaxations}

In this section we derive a hierarchy of relaxations for energy minimization. We model the space containing the particles by a compact metric space $(V, d)$, and assume the pair potential is given by a continuous function $F \colon (0, \mathrm{diam}(V)] \to \R$ with $F(s) \to \infty$ as $s \downarrow 0$. This natural assumption is made for convenience, as it avoids having to work with multisets. We denote the number of particles in the system by $N$. The \emph{ground state energy} is given by the minimum of
\[
\sum_{1 \leq i < j \leq N} F(d(x_i,x_j))
\]
over all sets $\{x_1,\ldots,x_N\}$ of $N$ distinct points in $V$. 

To compactify this problem, which will be convenient when we discuss duality, we introduce a graph that allows us to discard some configurations that are already known to be suboptimal. Let $B$ be an upper bound on the minimal energy. Such a number can be obtained by computing the energy of an arbitrary configuration of $N$ distinct points. Let $G$ be the graph with vertex set $V$, where distinct vertices $x$ and $y$ are adjacent whenever $F(d(x,y)) > B$. This ensures that the optimal $N$ point configurations are among the \emph{independent sets} of this graph, which are the subsets of the vertices for which no two vertices are adjacent. Alternatively we can use a result such as \cite{KuijlaarsSaffSun2007}, which states that for certain energy potentials the points of an optimal $N$-point configuration must be at least a certain distance $D$ apart form each other, and we can define the graph $G$ by letting $x$ and $y$ be adjacent whenever $d(x,y) < D$. Since this has the potential to exclude more configurations this approach might lead to stronger relaxations.

Let $I_t$ be the set of independent sets of cardinality at most $t$. Equip $V^t$ with the product topology and $\smash{V^t / q}$ with the quotient topology, where $q$ is the map that sends an ordered tuple $(x_1,\ldots,x_t)$ to the unordered set $\{x_1,\ldots,x_t\}$. The set $I_t \setminus \{\emptyset\}$ obtains a topology by viewing it as a subset of $\smash{V^t / q}$. Equip $I_t = I_t \setminus \{\emptyset\}$ with the disjoint union topology, so that $\emptyset$ is an isolated point in $I_t$. Let $I_{=t}$ be the set of independent sets of cardinality exactly $t$, and endow $I_{=t}$ with a topology as a subset of $I_t$. For each $t$, the topology on $I_t$ is the topology induced from the Hausdorff distance
\[
d_{\mathrm H}(J,J') = \max\big\{\max_{x \in J} \min_{y \in J'} d(x,y),\, \max_{y \in J'} \min_{x \in J} d(x,y)\big\}
\]
between sets $J,J' \in I_{=t}$. Intuitively, two points $J$ and $J'$ in $I_{=t}$ are close if $J'$ can be obtained from $J$ by a small perturbation of the points in $J$.

The graph $G$ as defined above is an example of a compact topological packing graph as defined in \cite{LaatVallentin2014}. A \emph{topological packing graph} is a graph whose vertex set is a Hausdorff topological space where each clique is contained in an open clique. Here a clique is a subset of the vertex set where each pair of distinct vertices is adjacent. Using the topological packing graph condition it can be shown that $I_t$ and $I_{=t}$ are compact metric spaces, and from the definition of $G$ it follows that $I_{=N}$ is nonempty.

Let $\Ccal(I_N)$ be the space continuous functions $I_N \to \R$ and define $f \in \Ccal(I_N)$ by
\[
f(S) = \begin{cases} 
F(d(x,y)) & \text{if } S = \{x,y\} \text{ with } x \neq y,\\
0 & \text{otherwise.}
\end{cases}
\]
The continuity of $f$ follows from the continuity of $F$ and the fact that $I_{=2}$ is both open and closed in $I_N$ (see \cite{LaatVallentin2014}). Let $\chi_S$ be the measure 
\begin{equation}
\label{eq:chi}
\chi_S = \sum_{R \subseteq S} \delta_R,
\end{equation}
where $\delta_R$ is the Dirac point measure at $R$. For $S = \{x_1,\ldots,x_N\}$ we have 
\[
\chi_S(f) = \sum_{R \subseteq S} f(R) = \sum_{1 \leq i<j\leq N} F(d(x_i,x_j)).
\]
This means that the ground state energy can be written as
\begin{equation}\label{eq:groundstate}
E = \min_{S \in I_{=N}} \chi_S(f),
\end{equation}
where the minimum is attained because the function  $S \mapsto \chi_S(f)$ is continuous and the space $I_N$ is compact.

To obtain energy lower bounds we construct a hierarchy $E_1, E_2, \ldots, E_N$ of relaxations of the above problem. These are minimization problems such that for each feasible solution of $E$ (a configuration of $N$ particles) we can immediately construct a feasible solution of $E_t$ having the same objective value. The problems $E_t$ have the important feature that we can give their dual optimization problems in an explicit form, which is crucial for performing computations, and that we can prove the duality gap to be zero. In Section~\ref{sec:energy convergence to the optimal energy} we show that the $N$th step $E_N$ in this hierarchy gives the ground state energy, and that the extreme points of the feasible set of $E_N$ are precisely the measures $\chi_S$ with $S \in I_{=N}$. That is, we show $E_N$ is a \emph{sharp} relaxation of $E$. The problems $E_t$ are convex optimization problems, and we say $E_N$ is a \emph{convexification} of $E$.

Denote the space of signed Radon measures on $I_{2t}$ by $\Mcal(I_{2t})$. Given an independent set $S \in I_{=N}$, let
\[
\lambda_S := \sum_{R \subseteq S : |R| \leq 2t} \delta_R.
\]

In the $t$-th step $E_t$ we will optimize over measures $\lambda \in \Mcal(I_{2t})$ that we require to satisfy three  properties that are satisfied by $\lambda_S$. The first of these properties is that $\lambda_S$ is a positive measure. The second property is that 
\[
\lambda_S(I_{=i}) = \sum_{R \subseteq S : |R| \leq 2t} \delta_R(I_{=i}) = |\{R \subseteq S : |R| = i\}| =  \binom{N}{i} \quad \text{for all} \quad 0 \leq i \leq 2t,
\]
where we use the convention $\binom{N}{i} = 0$ for $i > N$. 

The third property is more subtle. The measure $\lambda_S$ satisfies a moment condition. We use the moment techniques from \cite{LaatVallentin2014} to define what we mean by this. Define the operator
\[
A_t \colon \Ccal(I_t \times I_t)_\mathrm{sym} \to \Ccal(I_{2t}), \, A_tK(S) = \sum_{\substack{J,J' \in I_t \\ J \cup J' = S}} K(J,J'),
\]
where $\Ccal(I_t \times I_t)_\mathrm{sym}$ is the space of symmetric, continuous functions $I_t \times I_t \to \R$, which we call \emph{symmetric kernels}. A kernel $K$ is said to be \emph{positive definite} if the matrices $\smash{(K(J_i,J_j))_{i,j=1}^n}$ are positive semidefinite for all $n\in \N$ and $J_1,\ldots,J_n \in I_t$. The positive definite kernels form a convex cone which we denote by $\Ccal(I_t \times I_t\smash{)_{\succeq 0}}$. By the Riesz representation theorem, the topological duals of $\Ccal(I_t \times I_t\smash{)_\mathrm{sym}}$ and $\Ccal(I_{2t})$ can be identified with the spaces $\smash{\Mcal(I_t \times I_t)_\mathrm{sym}}$ and $\Mcal(I_{2t})$. Here $\Mcal(I_t \times I_t\smash{)_\mathrm{sym}}$ consists of the symmetric Radon measures, which are the measures $\mu$ that satisfy
\[
\mu(E \times F) = \mu(F \times E) \quad \text{for all Borel sets} \quad E, F \subseteq I_t.
\]
The dual cone of $\Ccal(I_t \times I_t)_{\succeq 0}$ is defined by
\[
\Mcal(I_t \times I_t)_{\succeq 0} = \Big\{ \mu \in \Mcal(I_t \times I_t)_{\mathrm{sym}} : \mu(K) \geq 0 \text{ for all } K \in \Ccal(I_t \times I_t)_{\succeq 0}\Big\},
\]
and we call the elements in this cone \emph{positive definite measures}.

We have the adjoint operator 
\[
A_t^* \colon \Mcal(I_{2t}) \to \Mcal(I_t \times I_t)_\mathrm{sym},
\]
which is defined by $A_t^*\lambda(K) = \lambda(A_t K)$ for all $\lambda \in \Mcal(I_{2t})$ and $K \in \Ccal(I_t \times I_t)_\mathrm{sym}$, and we use this dual operator and the cone of positive definite measures to define the moment condition on $\lambda$.
\begin{defin}
\label{energy:def:positive type}
A measure $\lambda \in \Mcal(I_{2t})$ is of positive type if 
\[
A_t^* \lambda \in \Mcal(I_t \times I_t)_{\succeq 0}.
\]
\end{defin}
\noindent
See \cite{LaatVallentin2014}*{Remark 1} for an explanation why we use the term positive type here. The measure $\lambda_S$ defined above is of positive type: For each $K \in \Ccal(I_t \times I_t)_{\succeq 0}$, we have
\[
A_t^*\lambda_S(K) = \sum_{R \subseteq S} \; \sum_{\substack{J,J' \in I_t \\ J \cup J' = R}} K(J,J') = \sum_{\substack{J,J' \in S \\ |J|,|J'| \leq t}} K(J,J') \geq 0.
\]

We define the $t$-th step in our hierarchy by optimizing over measures $\lambda \in \Mcal(I_{2t})$ satisfying the three properties  discussed above; that is, $\lambda$ is a positive measure, $\lambda(I_{=i}) = \binom{N}{i}$ for $0\leq i \leq 2t$, and $A_t^* \lambda \in \mathcal M(I_t \times I_t)_{\succeq 0}$.
\begin{defin}
For $1 \leq t \leq N$, let
\begin{align*}
E_t := \min \Big\{ \lambda(f) : \; &  \lambda \in \Mcal(I_{2t}) \text{ positive and of positive type},\\[-0.5em]
& \lambda(I_{=i}) = \tbinom{N}{i} \text{ for } 0 \leq i \leq 2t\Big\}.
\end{align*}
\end{defin}
In Section~\ref{sec:generalized duality} we prove strong duality, which implies the minimum here is attained. By construction, the measure $\lambda_S$ is feasible for $E_t$, so $E_t \leq E$ for all $t$. A similar argument shows $E_{t} \leq E_{t+1}$ for all $t$. In Section~\ref{sec:energy convergence to the optimal energy}
 we prove $E_N = E$.

\section{Connection to the Lasserre hierarchy}
\label{sec:energy-lasserre}

We originally came up with the hierarchy $E_1,E_2,\ldots,E_N$ by first applying a variation of the Lasserre hierarchy from polynomial optimization to energy minimization problems where $V$ is finite. We then reformulate (and possibly weaken, but still preserving convergence) the constraints in the resulting relaxations in such a way that we can find a useful generalization to the case where $V$ is infinite.

A \emph{polynomial optimization problem} is a problem of the form
\[
\inf \Big\{ p(z) : z \in \R^n, \, g_j(z) \geq 0 \text{ for } j \in [m]\Big\},
\]
where $p, g_1,\ldots,g_m \in \R[z_1,\ldots,z_n]$ and $[m] = \{1,\ldots,m\}$. In general, both finding the global minimum and proving optimality of a feasible solution $z$ are difficult problems. A powerful and popular approach to obtain lower bounds on the global minimum is to use the moment techniques and dual sum-of-squares techniques by Lasserre~\cite{Lasserre2000} and Parrilo~\cite{Parrilo2000}, which give a sequence of increasingly strong semidefinite programming relaxations. 

We are interested in \emph{binary} polynomial optimization problems, which are problems of the form
\[
\inf \Big\{ p(z) : z \in \{0,1\}^n, \, g_j(z) \geq 0 \text{ for } j \in [m] \Big\}.
\]
These problems are special cases of polynomial optimization problems, because we can enforce the constraint $z_i \in \{0,1\}$ by adding the polynomial constraints $z_i(1-z_i) \geq 0$ and $-z_i(1-z_i) \geq 0$.
If we assume the space $V$ of an energy minimization problem to be finite, say, $V = [n]$ for some $n \in \N$, then we can write the energy minimization problem as the binary polynomial optimization problem
\[
\min \Big\{ \sum_{1 \leq i < j \leq n} F(\{i,j\}) z_i z_j : z \in \{0, 1\}^n, \, \kappa(z) \ge 0, -\kappa(z) \ge 0 \Big\}, 
\]
where 
\[
\kappa(z) := \sum_{i=1}^n z_i - N
\]
and $F$ is the pair potential as defined in Section~\ref{sec:derivation of the hierarchy via relaxations}.

Let $G$ be the graph with vertex set $[n]$ and no edges. Then $I_t$ is the set of all subsets of $[n]$ of cardinality at most $t$. In a binary polynomial optimization problem we may assume the polynomials to be square free, so that we can write
\[
p(z) = \sum_{S \in I_{\deg(p)}} p_S z^S, \quad \text{where} \quad z^S := \prod_{i \in S} z_i.
\]

Given an integer $t \in \N$ and a vector $y \in \R^{I_{2t}}$, define the $t$-th \emph{moment matrix} $M_t(y) \in \R^{I_t \times I_t}$ by $M_t(y)_{J, J'} = y_{J \cup J'}$. We define the $t$-th \emph{localizing matrix} with respect to a polynomial $g \in \R[x_1,\ldots,x_n]$ to be the partial matrix $M_t^g(y)$, which has the same row and column indices as $M_t(y)$, where the $(J,J')$-entry is set to
\[
\sum_{R \in I_{\deg(g)}} y_{J \cup J' \cup R}\, g_R
\]
whenever $|J \cup J'| \leq 2t-\deg(g)$, and left unspecified otherwise. 
By $M_t^p(y) \succeq 0$ we mean that $y$ is a vector such that $M_t^p(y)$ can be completed to a positive semidefinite matrix (this is a semidefinite constraint on $y$). 
Using these definitions we define for $t \geq \deg(p)$ the following semidefinite programming relaxation of the binary polynomial optimization problem  given above:
\begin{align*}
\inf \Big\{ \sum_{S \in I_{\deg(p)}} p_S y_S : \; & y \in \R_{\geq 0}^{I_{2t}}, \, y_\emptyset = 1, \; M_t(y) \succeq 0,\\[-1em]
& M_t^{g_j}(y) \succeq 0 \text{ for } j \in [m]\Big\}.
\end{align*}

These relaxations were introduced by Lasserre in \cite{Lasserre2002a}. The only modifications we make here is that we restrict $y$ to be nonnegative and that we take the localizing matrices $M_t^{g_j}(y)$ to be partial matrices indexed by $I_t$, instead of full matrices indexed by $I_{t-\lceil \deg(g_j)/2 \rceil}$. These make the semidefinite programs only slightly more difficult to solve, but in some cases, such as the case of energy minimization as discussed here, it leads to much stronger bounds.

In the binary polynomial optimization formulation for energy minimization we have two polynomial constraints: $\kappa(x) \geq 0$ and $-\kappa(x) \geq 0$. So, in the relaxation we have the constraints $M_t^\kappa(y) \succeq 0$ and $-M_t^\kappa(y) = M_t^{-\kappa}(y) \succeq 0$. This reduces to the constraint $M_t^\kappa(y) = 0$; that is, all specified entries of $M_t^\kappa(y)$ are required to be zero. These constraints reduce to the linear constraints
\[
N y_S = \sum_{j=1}^n y_{S \cup \{j\}} \quad \text{for all} \quad S \in I_{2t-1}.
\]
So, for energy minimization we get the relaxations
\begin{align*}
L_t := \inf \Big\{ \sum_{1 \leq i < j \leq n} F(\{i,j\}) \,y_{\{i,j\}} : \; & y \in \R_{\geq 0}^{I_{2t}}, \, y_\emptyset = 1, \; M_t(y) \succeq 0,\\[-1.3em]
& N y_S = \sum_{j=1}^n y_{S \cup \{j\}} \quad \text{for} \quad S \in I_{2t-1}\Big\}.
\end{align*}

The linear constraints in these problems become problematic when we want to generalize $V$ from the finite set $[n]$ to an uncountable set. The reason for this is that in the infinite dimensional generalization we want to use measures $\lambda \in \Mcal(I_{2t})$ instead of vectors $y \in \R^{I_{2t}}$ (because this allows for a satisfying duality theory; see Section~\ref{sec:generalized duality}), and these then become uncountably many ``thin'' constraints on $\lambda$. By thin we mean that the constraints are of the form $\lambda(E) = b$, where $E$ is a set with empty interior, which means these constraints have no grip on the part of a measure that is zero on sets with empty interior.

In the following lemma we show these constraints imply $2t+1$ very natural constraints on $y$. In particular, this lemma implies that for a feasible solution $y$ of $L_t$, we have $y_S = 0$ for all $S \subseteq [n]$ with $|S| > N$. If we replace the linear constraints in $L_t$ by these induced constraints, then we obtain the problem $E_t$ as defined in the previous section for the case where $V = [n]$. Therefore, our construction $E_t$ for compact $V$ extends a (possible) weakening of the bound $L_t$.

\begin{lemma}
Let $t \in \N_0$ and $y \in \R^{I_{2t}}$. If
\[
y_\emptyset = 1 \quad \text{and} \quad N y_S = \sum_{j=1}^n y_{S \cup \{j\}} \quad \text{for all} \quad S \in I_{2t-1},
\]
then
\[
\sum_{S \in I_{=i}} y_S = \binom{N}{i} \quad \text{for all} \quad 0 \leq i \leq 2t.
\] 
\end{lemma}
\begin{proof}
For $i=0$ we have 
\[
\sum_{S \in I_{=i}} y_S = y_\emptyset = 1 = \binom{N}{i}.
\]
If $\sum_{S \in I_{=i-1}} y_S = \binom{N}{i-1}$ for some $0 \leq i \leq 2t-1$, then
\begin{align*}
\sum_{S \in I_{=i}} y_S 
&= \frac{1}{i} \sum_{S \in I_{=i-1}} \sum_{j \in [n] \setminus S} y_{S \cup \{j\}} 
= \frac{1}{i} \sum_{S \in I_{=i-1}} \left( \sum_{j = 1}^n y_{S \cup \{j\}} - |S| y_S \right)\\ 
&= \frac{1}{i} \sum_{S \in I_{=i-1}} \big(Ny_S - (i - 1) y_S\big) = \frac{1}{i} \sum_{S \in I_{=i-1}} (N - i + 1) y_S\\
&= \frac{N - i + 1}{i} \sum_{S \in I_{=i-1}} y_S = \frac{N - i + 1}{i} \binom{N}{i-1} = \binom{N}{i}.
\end{align*}
Hence, the proof follows by induction.
\end{proof}

\section{Convergence to the ground state energy}
\label{sec:energy convergence to the optimal energy}

In this section we show that the hierarchy $\{E_t\}$ converges to the optimal energy $E$ in at most $N$ steps. Moreover, the extreme points of the feasible set of $E_N$ are precisely the measures $\chi_S$ with $S \in I_{=N}$. These results follow from the following proposition, whose proof follows directly from the proof of \cite{LaatVallentin2014}*{Proposition 4.1}.
\begin{proposition}
\label{prop:lambda to sigma}
For each measure $\lambda \in \Mcal(I_{2t})$ there exists a unique measure $\sigma \in \Mcal(I_{2t})$ such that $\lambda = \int \chi_S \, d\sigma(S)$. If $\lambda$ is supported on $I_t$ and is of positive type, that is, $A_t^* \lambda \in \Mcal(I_t \times I_t)_{\succeq 0}$, then $\sigma$ is a positive measure supported on $I_t$.
\end{proposition}

In the above proposition $\lambda$ is defined using a vector-valued integral, where the notation $\lambda = \int \chi_S \, d\sigma(S)$ means
\[
\lambda(g) = \int \chi_S(f) \, d\sigma(S) \quad \text{for all} \quad g \in \Ccal(I_N).
\]

Using this proposition we can prove the convergence result:

\begin{proposition}
The $N$th step $E_N$  gives the optimal energy $E$.
\end{proposition}
\begin{proof}
Let $\lambda \in \Mcal(I_{2N})$ be feasible for $E_N$. We have $\lambda \geq 0$ and $\lambda(I_{=i}) = \tbinom{N}{i} = 0$ for $i > N$, so $\lambda$ is supported on $I_N$. Since $\lambda$ is also of positive type, by Proposition~\ref{prop:lambda to sigma} there exists a positive measure $\sigma \in \Mcal(I_N)$ such that $\lambda = \int \chi_S \, d\sigma(S)$. 
We have 
\[
1 = \tbinom{N}{0} = \lambda(\{\emptyset\}) = \int \chi_S(\{\emptyset\}) \, d\sigma(S) = \int d\sigma = \sigma(I_N),
\]
so $\sigma$ is a probability measure. Moreover, 
\[
1 = \tbinom{N}{N} = \lambda(I_{=N}) = \int \chi_S(I_{=N}) \, d\sigma(S) = \sigma(I_{=N}),
\]
so $\sigma$ is supported on $I_{=N}$. The objective value of $\lambda$ is given by 
\[
\lambda(f) = \int \chi_S(f)\,d\sigma(S)  \geq \int E \,d\sigma = E,
\]
where the inequality follows since $\chi_S(f) \geq E$ for all $S \in I_{=N}$.
It follows that $E_N \geq E$. Since we already known $E_N \leq E$, this completes the proof. 
\end{proof}

Using the ideas of the above proof together with the proof of \cite{LaatVallentin2014}*{Proposition 4.1} it follows that the extreme points of the feasible set of $E_N$ are precisely the measures $\chi_S$ with $S \in I_{=N}$.

\section{Optimization with infinitely many binary variables}
\label{sec:infinite dimensional optimization}

We discuss the duality theory and symmetry reduction for a more general type of problems that arise naturally when forming moment relaxations of optimization problems with infinitely many binary variables. This includes the moment relaxations for both energy minimization and packing problems. Although there are infinitely many variables, we assume that in a feasible solution only finitely many of them are active (nonzero) at the same time, and active variables cannot be too close. For this we assume $G = (V,E)$ to be a compact \emph{topological packing graph} as discussed in Section~\ref{sec:derivation of the hierarchy via relaxations}.
 
\begin{defin}
	\label{def:generalbinary}
Let $G$ be a topological packing graph. Given integers $t$ and $m$, functions $g, g_1,\ldots, g_m \in \Ccal(I_{2t})$, and scalars $b_1,\ldots,b_m \in \R$, we define the optimization problem $H = H_{G,t}^\mathrm{inf}(g;g_1,\ldots,g_m;b_1,\ldots,b_m)$ by
\[
H = \inf \Big\{ \lambda(g) : \lambda \in \Mcal(I_{2t})_{\geq 0}, \, A_t^*\lambda \in \Mcal(I_t \times I_t)_{\succeq 0}, \, \lambda(g_j) = b_j \text{ for } j \in [m] \Big\}.
\]
\end{defin}
For energy minimization we have 
\[
E_t = H_{G,t}^\mathrm{min}(F;1_{I_{=0}},\ldots,1_{I_{=2t}};\tbinom{N}{0},\ldots,\tbinom{N}{2t}),
\]
where $G$ and $F$ are the graph and potential function as defined in Section~\ref{sec:derivation of the hierarchy via relaxations}. For packing problems in discrete geometry, the $t$-th step of the hierarchy from \cite{LaatVallentin2014} is given by $H_{G,t}^\mathrm{max}(1_{I_{=1}}; 1_{\{\emptyset\}}, 1)$, where $G$ is the topological packing graph defining the packing problem.

\subsection{Duality}
\label{sec:generalized duality}

The optimization problem $H$ is a conic program over the cone 
\[
\Mcal(I_t \times I_t)_{\succeq 0} \times \Mcal(I_{2t})_{\geq 0},
\]
where we refer to \cite{Barvinok2002} for an introduction to conic programming. If we endow both $\Mcal(I_t \times I_t)_\mathrm{sym}$ and $\Mcal(I_{2t})$ with the weak* topologies, then the topological dual spaces can be identified with $\Ccal(I_t \times I_t)_\mathrm{sym}$ and $\Ccal(I_{2t})$. The tuples $(\Ccal(I_{2t}), \Mcal(I_{2t}))$ and $(\Ccal(I_t \times I_t)_\mathrm{sym}, \Mcal(I_t \times I_t)_\mathrm{sym})$ are dual pairs, and  the dual pairings 
\[
\langle g, \lambda) = \lambda(g) = \int_{I_t} g(S) \, d\lambda(S) \quad \text{and} \quad \langle K, \mu \rangle = \mu(K) = \int_{I_t \times I_t} K(J,J') \, d\mu(J,J')
\]
are nondegenerate, by which we mean that $\lambda(g) = \mu(K) = 0$ for all $g \in \Ccal(I_{2t})$ and $K \in \Ccal(I_t \times I_t)_{\mathrm{sym}}$ implies $\lambda=0$ and $\mu = 0$. The dual cones are then given by $\Ccal(I_t \times I_t)_{\succeq 0}$ and $\Ccal(I_{2t})_{\geq 0}$, and by conic duality we obtain the dual conic program
\[
H^* = \sup \Big\{ \sum_{i=1}^m b_i a_i : a \in \R^m, \, K \in \Ccal(I_t \times I_t)_{\succeq 0}, \, g - \sum_{i=1}^m a_i g_i - A_tK \in \Ccal(I_{2t})_{\geq 0} \Big\}.
\]
By weak duality we have $H^* \leq H$. The following theorem, which is a generalization of the results in \cite{LaatVallentin2014}*{Chapter 3}, gives a sufficient condition for strong duality.
\begin{theorem}
\label{thm:strong duality of generalized hierarchy}
If $H$ admits a feasible solution, and if the set 
\begin{equation*}\label{eq:mustbe0}
\Big\{ \lambda \in \Mcal(I_{2t})_{\geq 0} : A_t^*\lambda \in \Mcal(I_t \times I_t)_{\succeq 0}, \, \lambda(g) = \lambda(g_1) = \cdots = \lambda(g_m) = 0 \Big\},
\end{equation*}
is trivial; that is, contains only the zero measure $\lambda = 0$, then strong duality holds: $H = H^*$ and the minimum in $H$ is attained.
\end{theorem}
\begin{proof}
To show that strong duality holds we use the closed cone condition as described in  \cite[Chapter IV.7]{Barvinok2002}. This closed cone condition says that if $H$ admits a feasible solution, and the cone 
\[
K = \big\{ (A_t^* \lambda - \mu, \lambda(g_1), \ldots, \lambda(g_m), \lambda(g)) : \lambda \in \Mcal(I_{2t})_{\geq 0}, \, \mu \in \Mcal(I_t \times I_t)_{\succeq 0} \big\}
\]
is closed in $\Mcal(I_t \times I_t)_{\succeq 0} \times \R^m \times \R$, then strong duality holds: $H = H^*$ and the minimum in $H$ is attained. 

This cone $K$ decomposes as the Minkowski difference $K = K_1 - K_2$, with
\[
K_1 = \{ (A_t^* \lambda, \lambda(g_1), \ldots, \lambda(g_m), \lambda(g)) : \lambda \in \Mcal(I_{2t})_{\geq 0} \}
\]
and
\[
K_2 = \{ (\mu, 0, 0) : \mu \in \Mcal(I_t \times I_t)_{\succeq 0} \}.
\]

By Klee~\cite{Klee1955} and Dieudonn\'e~\cite{Dieudonne1966}, a sufficient condition for the cone $K$ to be closed is that $K_1 \cap K_2 = \{0\}$, $K_1$ is closed and locally compact, and $K_2$ is closed. The first condition $K_1 \cap K_2 = \{0\}$ follows immediately from the hypothesis of the theorem. In \cite{LaatVallentin2014} it is shown that $K_1$ is closed and locally compact. That $K_2$ is closed follows immediately from $\Mcal(I_t \times I_t)_{\succeq 0}$ being closed.
\end{proof}

In \cite{LaatVallentin2014}*{Lemma 3.1.5} it is shown that the set 
\begin{equation*}\label{eqset0}
\big\{ \lambda \in \Mcal(I_{2t})_{\geq 0} : A_t^*\lambda \in \Mcal(I_t \times I_t)_{\succeq 0}, \, \lambda(\{\emptyset\}) = 0 \big\}
\end{equation*}
is trivial, which means Theorem~\ref{thm:strong duality of generalized hierarchy} applies to the case where each $\lambda \in \Mcal(I_{2t})_{\geq 0}$ with $\lambda(g) = \lambda(g_1) = \cdots = \lambda(g_m) = 0$ satisfies $\lambda(\{\emptyset\}) = 0$. The programs $E_t$ are such examples, since $g_1$ here is the indicator function of $I_{=0}$, so $\lambda(g_1) = 0$ implies $\lambda(\{\emptyset\})  = 0$. For each $t$, the program $E_t$ admits a feasible solution (see Section~\ref{sec:derivation of the hierarchy via relaxations}). Thus by Theorem~\ref{thm:strong duality of generalized hierarchy} strong duality holds for the pair $(E_t, E_t^*)$.

\subsection{Symmetry reduction}
\label{sec:energy:symmetry reduction}

Given a compact group $\Gamma$ with a continuous action on the vertex set $V$ of a compact topological packing graph $G$, we say the optimization problem $H_{G,t}^{\inf}(g; g_1, \ldots, g_m; b_1,\ldots,b_m)$ is \emph{$\Gamma$-invariant} if
\begin{enumerate}
\item the adjacency relations in $G$ are invariant under the action of $\Gamma$, so that the action extends to a continuous action on $I_N$ given by $\gamma \emptyset = \emptyset$ and $\gamma\{x_1,\ldots,x_N\} = \{\gamma x_1,\ldots,\gamma x_N\}$;
\item the functions $g,g_1,\ldots,g_m$ are $\Gamma$-invariant. 
\end{enumerate}   
Using this definition, the relaxations $E_t = H_{G,t}^\mathrm{min}(F;1_{I_{=0}},\ldots,1_{I_{=2t}};\tbinom{N}{0},\ldots,\tbinom{N}{2t})$ are $\Gamma$-invariant whenever $(V,d)$ is $\Gamma$-invariant.

We use this symmetry to restrict to invariant variables in both the primal and dual optimization problems. To make the best use of the symmetry we should take $\Gamma$ as large as possible. For energy minimization problems this means we should take it to be the symmetry group of the metric space $(V,d)$, which for the Thomson problem means we take $\Gamma = O(3)$.

Let $\Ccal(I_{2t})^\Gamma$ be the subspace consisting of $\Gamma$-invariant functions, and $\Ccal(I_t \times I_t)_\mathrm{sym}^\Gamma$ the subspace of $\Gamma$-invariant symmetric kernels, and define the cones 
\[
\Ccal(I_{2t})_{\geq 0}^\Gamma = \Ccal(I_{2t})_{\geq 0} \cap \Ccal(I_{2t})^\Gamma \quad  \text{and}  \quad \Ccal(I_t \times I_t)_{\succeq 0}^\Gamma = \Ccal(I_t \times I_t)_{\succeq 0} \cap \Ccal(I_t \times I_t)_\mathrm{sym}^\Gamma.
\]

Given a function $g \in \Ccal(I_{2t})$, we define its symmetrization $\bar g \in \Ccal(I_{2t})^\Gamma$ by 
\[
\bar g(S) = \int_\Gamma g(\gamma S)\, d\gamma,
\]
where we integrate over the normalized Haar measure of $\Gamma$. Similarly, given a kernel $K \in \Ccal(I_t \times I_t)_{\mathrm{sym}}$, we define its symmetrization $\bar K$ by 
\[
\bar K(J, J') = \int_\Gamma K(\gamma J, \gamma J') \, d\gamma.
\]
Using these definitions we can define the symmetrizations $\bar\lambda$ and $\bar\mu$ of measures $\lambda \in \Mcal(I_{2t})$ and $\mu \in \Mcal(I_t \times I_t)_\mathrm{sym}$ by $\bar\lambda(g) = \lambda(\bar g)$ and $\bar\mu(K) = \mu(\bar K)$. It follows that the spaces $\Mcal(I_{2t})^\Gamma$ and $\Mcal(I_t \times I_t)_\mathrm{sym}^\Gamma$ can be identified with the topological duals of $\smash{\Ccal(I_{2t})^\Gamma}$ and $\smash{\Ccal(I_t \times I_t)_\mathrm{sym}^\Gamma}$, and as in the nonsymmetrized situation these form dual pairs.

We have
\[
A_tK(\gamma S) = \sum_{\substack{J,J' \in I_t \\ J \cup J' = \gamma S}} K(J,J') = \sum_{\substack{J,J' \in I_t \\ \gamma^{-1}J \cup \gamma^{-1}J' = S}} K(J,J') = \sum_{\substack{J,J' \in I_t \\ J \cup J' = S}} K(\gamma J, \gamma J'),
\]
so $A_t$ maps $\Gamma$-invariant kernels to $\Gamma$-invariant functions. This means we can view $A_t$ as an operator from $\smash{\Ccal(I_t \times I_t)_{\mathrm{sym}}^\Gamma}$ to $\Ccal(I_{2t})^\Gamma$, and we can view the adjoint $A_t^*$ as an operator from $\Mcal(I_{2t})^\Gamma$ to $\Mcal(I_t \times I_t)_\mathrm{sym}^\Gamma$.

We define the symmetrization of a $\Gamma$-invariant primal problem $H$ from Definition~\ref{def:generalbinary} by
\[
H_\Gamma = \min \Big\{ \lambda(g) : \lambda \in \Mcal(I_{2t})_{\geq 0}^\Gamma, \, A_t^* \lambda \in \Mcal(I_t \times I_t)_{\succeq 0}^\Gamma, \,\lambda(g_i) = b_i \text{ for } i \in [m] \Big\},
\]
and the symmetrization of the dual program $H^*$ by
\[
H_\Gamma^* = \sup \Big\{ \sum_{i=1}^m b_i a_i : a \in \R^m, \, K \in \Ccal(I_t \times I_t)_{\succeq 0}^\Gamma, \, g - \sum_{i=1}^m a_i g_i - A_tK \in \Ccal(I_{2t})_{\geq 0}^\Gamma \Big\}. 
\]
Given feasible solutions $\lambda$ and $K$ of $H$ and $H^*$, the symmetrizations $\bar\lambda$ and $\bar K$ are feasible for $H_\Gamma$ and $H_\Gamma^*$  and have the same objective values. This shows $H = H_\Gamma$ and $H^* = H_\Gamma^*$. As a result we have $H_\Gamma = H_\Gamma^*$, which alternatively could be shown by proving strong duality for the symmetrized problems.

In the following section we discuss the construction of a nested sequence $\{C_d\}_{d=0}^\infty$ of inner approximating cones of $\Ccal(I_t \times I_t)_{\succeq 0}^\Gamma$ such that the union $\cup_{d=0}^\infty C_d$ is uniformly dense in $\smash{\Ccal(I_t \times I_t)_{\succeq 0}^\Gamma}$. This sequence is constructed in such a way that optimization over $C_d$ is easies than optimization over the original cone, and gets more difficult as $d$ grows.
Define $\smash{H_d^*}$ to be the optimization problem $\smash{H_\Gamma^*}$ with the cone of invariant, positive definite kernels replaced by its inner approximation $C_d$. Now we give a sufficient condition for the programs $H_d^*$ to approximate the program $H_\Gamma^*$. The following proposition applies for the case of $H = E_t$ by selecting $c = 0$ and $y = -e$, where $e$ is the all-ones vector. This shows that if we let $E_{t,d}^*$ be the problem $E_t^*$ with the cone $\Ccal(I_t \times I_t)_{\succeq 0}$ replaced by $C_d$, then $E_{t,d}^* \to E_t$ as $d \to \infty$.

\begin{proposition}
If there exists a scalar $c \in \R$ and a vector $y \in \R^m$ for which $c g - \sum_{i=1}^m y_i g_i$ is a strictly positive function, then $H_d^* \to H^*$ as $d \to \infty$.
\end{proposition}
\begin{proof}
Select $c \in \R$ and $y \in \R^m$ for which $c g - \sum_{i=1}^m y_i g_i$ is a strictly positive function. Let $(a, K)$ be a feasible solution of $H_\Gamma^*$ and let $\varepsilon > 0$. Let 
\[
\kappa = \min_{S \in I_{2t}} \left(c g(S) - \sum_{i=1}^m y_i g_i(S)\right),
\]
where the minimum is attained and strictly positive because we optimize a continuous function over a compact set. We have $g - \sum_{i=1}^m a_i g_i - A_t K \geq 0$, so, for any $\delta \geq 0$, we have 
\[
g + \delta c g - \sum_{i=1}^m (a_i + \delta y_i) g_i - A_t K \geq \delta\kappa,
\]
and hence
\[
g - \sum_{i=1}^m \frac{a_i+\delta y_i}{1+\delta c} g_i - A_t \Big(\frac{1}{1+\delta c}K\Big) \geq \frac{\delta\kappa}{1+\delta c}.
\]
Since $\bigcup_{d=0}^\infty C_d$ is uniformly dense in $\Ccal(I_t \times I_t)_{\succeq 0}^\Gamma$, and since $A_t$ is a bounded operator, there exists a $d_\delta \in \N_0$ and a kernel $L_\delta \in C_{d_\delta}$ such that 
\[
\left\|A_t \Big(\frac{1}{1+\delta c}K\Big) - A_t L_\delta \right\|_\infty \leq \frac{\delta\kappa}{1+\delta c}.
\]
This means that 
\[
f - \sum_{i=1}^m \frac{a_i+\delta y_i}{1+\delta c} g_i - A_t L_\delta \geq 0.
\]
So, for all $\delta > 0$, $((a+\delta y)/(1+\delta c), L_\delta)$ is feasible for $H_{d_\delta}^*$, and as $\delta \downarrow 0$, its objective value goes to the objective value of $(a,K)$. This shows $H_d^* \to H^*$ as $d \to \infty$.
\end{proof}

\section{Approximating the cone of invariant positive definite kernels}
\label{chap:energy:sec:explicit symmetry}

In this section we show how to approximate the cone of invariant positive definite kernels by a sequence $C_1,C_2,\ldots$ of simpler inner approximating cones. This sequence converges in the sense that the union of the inner approximating cones is uniformly dense in the cone of invariant positive definite kernels, and each $C_d$ is isomorphic to a finite product of positive semidefinite matrix cones (their elements are said to be block diagonalized). 

In Section~\ref{subsecSymmetry adapted systems and zonal matrices} we first give some background on symmetry adapted systems and zonal matrices and how these can be used to construct inner approximating cones. Then we use the results from Appendix~\ref{sec:invariant posdef kernels} to show the existence of an inner approximating sequence as described above for the cone $\smash{\Ccal(X \times X)_{\succeq 0}^\Gamma}$, where $X$ is a compact metric space and $\Gamma$ a compact group with a continuous action on $X$. 

In Section~\ref{sect:harmonci analysis on subset spaces} we show how such a sequence can be constructed explicitly for the case where $X=I_t$, assuming we have certain knowledge about the representation theory of $V$. In Section~\ref{subsect:explicit harmonic analysis} we then perform the construction explicitly for $X = I_2$, $V = S^2$, and $\Gamma = O(3)$. This explicit construction is later used to compute $E_2^*$.

\subsection{Symmetry adapted systems and zonal matrices}
\label{subsecSymmetry adapted systems and zonal matrices}

In this section we show how to define the inner approximating cones by using symmetry adapted systems and zonal matrices. Since we  use representation theory, it is convenient to work over the complex numbers and first consider the cone $\Ccal(X \times X; \C)_{\succeq 0}^\Gamma$ of \emph{Hermitian}, $\Gamma$-invariant, positive definite kernels. Here a Hermitian kernel $K \in \Ccal(X \times X; \C)$ is said to be \emph{positive definite} if 
\[
\sum_{i,j=1}^n c_i \overline{c_j} K(x_i,x_j) \geq 0 \quad \text{for all}\quad n \in \N, x \in X^n, c \in \C^n.
\]

We start by giving the definition of a symmetry adapted system for $X$. Let $\mu$ be a Radon measure on $X$ that is strictly positive and $\Gamma$-invariant. By \emph{strictly positive} we mean that $\mu(U) > 0$ for all open sets $U$ in $X$, and by $\Gamma$-invariant we mean $\mu(\gamma U) = \mu(U)$ for all $\gamma \in \Gamma$ and all Borel sets $U$ in $X$. Such a measure always exists (see Lemma~\ref{lem:existence of strictly positive measure}). We define an \emph{orthonormal system} of $X$ to be a set consisting of continuous, complex-valued functions on $X$ that are orthonormal with respect to the $L^2(X, \mu; \C)$ inner product
\[
\langle f, g \rangle = \int_X f(x) \overline{g(x)} \, d\mu(x).
\]
Such a system is said to be \emph{complete} if its span is uniformly dense in the space $\Ccal(X; \C)$ of continuous complex-valued functions on $X$.

To define what it means for such a system to be symmetry adapted we need some representation theory. A \emph{unitary representation} of $\Gamma$ is a continuous group homomorphism from $\Gamma$ to the group $U(\Hcal)$ of unitary operators on a nontrivial Hilbert space $\Hcal$, where $U(\Hcal)$ is equipped with the weak (or strong, they are the same here) operator topology. Such a representation is said to be \emph{irreducible} when $\Hcal$ does not admit a nontrivial closed invariant subspace. Two unitary representations $\pi_1 \colon \Gamma \to U(\Hcal_1)$ and $\pi_2 \colon \Gamma \to U(\Hcal_2)$ are \emph{equivalent} if there exists a unitary operator $T \colon \Hcal_1 \to \Hcal_2$ that is $\Gamma$-equivariant; that is, $T\pi_1(\gamma)u = \pi_2(\gamma)Tu$ for all $\gamma \in \Gamma$ and $u \in \Hcal_1$.  Let $\smash{\hat\Gamma}$ be a complete set of inequivalent irreducible unitary representations of $\Gamma$, and denote the dimension of a representation $\pi \in \smash{\hat \Gamma}$ by $d_\pi$. A particularly important example of a unitary representation is given by
\[
L \colon \Gamma \to U(L^2(X, \mu; \C)), \, L(\gamma)f(x) = f(\gamma^{-1} x).
\]

A complete orthonormal system of $X$ is said to be a \emph{symmetry adapted system} of $X$ if there exist numbers $0 \leq m_\pi \leq \infty$ for which we can write the system as 
\[
\Big\{e_{\pi,i,j} : \pi \in\hat\Gamma, \, i \in [m_\pi], \, j \in [d_\pi]\Big\},
\]
with $\Hcal_{\pi,i} = \mathrm{span}\{e_{\pi,i,1}, \ldots, e_{\pi,i,d_\pi}\}$ equivalent to $\pi$ as a unitary subrepresentation of $L$, and where there exist $\Gamma$-equivariant unitary operators $T_{\pi,i,i'} \colon \Hcal_{\pi, i} \to \Hcal_{\pi,i'}$ with $\smash{e_{\pi, i', j} = T_{\pi,i,i'} e_{\pi,i,j}}$ for all $\pi$, $i$, $i'$, and $j$. In Theorem~\ref{thm:symmetry adapted system} we show such a system always exists. The number $m_\pi$ can be shown to be equal to the dimension of the space $\mathrm{Hom}_\Gamma(X, \Hcal_\pi)$ of $\Gamma$-equivariant, continuous functions from $X$ to $\Hcal_\pi$, where $\Hcal_\pi$ is the Hilbert space of the irreducible representation $\pi$. Hence $m_\pi$ does not depend on the choice of symmetry adapted system.

The spaces $H_{\pi, i}$ are pairwise orthogonal irreducible subrepresentations of $L$, each spanned by continuous functions, such that $\sum_{\pi \in \hat\Gamma} \smash{\sum_{i \in [m]}} H_{\pi,i}$ is uniformly dense in $\Ccal(X; \C)$. If, on the other hand, we are given a set of subspaces satisfying the above properties, then we can immediately construct a symmetry adapted system by simply selecting appropriate bases of the subspaces. 

Now we first show how the extreme rays of $\smash{\Ccal(X \times X; \C)_{\succeq 0}^\Gamma}$ suggest a ``simultaneous block diagonalization'' of the kernels in this cone. A nonzero vector $x$ in a cone $K$ lies on an \emph{extreme ray} if $x_1,x_2 \in \R_{\geq 0}x$ for all $x_1,x_2\in K$ with $x = x_1+x_2$. In Theorem~\ref{kernels:theorem:extreme rays} we show that a kernel $K \in \smash{\Ccal(X \times X; \C)_{\succeq 0}^\Gamma}$ lies on an extreme ray if and only if there exists an irreducible unitary representation $\pi \colon \Gamma \to U(\Hcal_\pi)$ and a function $\varphi \in \mathrm{Hom}_\Gamma(X, \Hcal_\pi)$ such that
\[
K(x,y) = \langle \varphi(x), \varphi(y) \rangle \quad \text{for all} \quad x,y \in X. 
\]
In the case where $m_\pi < \infty$, we have a finite basis $\varphi_1,\ldots,\varphi_{m_\pi}$ of $\mathrm{Hom}_\Gamma(X, \Hcal_\pi)$, and the map 
\[
A \mapsto \sum_{i,j=1}^{m_\pi} A_{i,j} \langle \varphi_i(\cdot), \varphi_j(\cdot) \rangle
\]
is an isomorphism from the cone of $m_\pi \times m_\pi$ Hermitian positive semidefinite matrices to the convex hull of the extreme rays of $\smash{\Ccal(X \times X; \C)_{\succeq 0}^\Gamma}$ corresponding to $\pi$. Moreover, if $\hat\Gamma$ is finite and $m_\pi$ is finite for all $\pi \in \smash{\hat\Gamma}$, this gives an isomorphism between a finite product of Hermitian positive semidefinite matrix cones and $\smash{\Ccal(X \times X; \C)_{\succeq 0}^\Gamma}$. For our applications, however, we are particularly interested in the situation where these numbers are not finite, and hence we need to consider convergence.

Given a symmetry adapted system $\{ e_{\pi,i,j} \}$ of $X$, we define the matrices
\[
E_\pi(x)_{i,j} = e_{\pi,i,j}(x) \quad \text{for} \quad \pi \in \hat\Gamma, x\in X, i \in [m_\pi], j \in [d_\pi],
\]
and we use this to define the \emph{zonal matrices}
\[
Z_\pi(x,y) = E_\pi(x) E_\pi(y)^* \quad \text{for}  \quad \pi \in \hat \Gamma \quad \text{and} \quad x,y\in X,
\]
where $E_\pi(y)^*$ is the conjugate transpose of $E_\pi(y)$.
The \emph{Fourier coefficients} of a kernel $K \in \Ccal(X \times X; \C)$ are defined by
\[
\hat K(\pi) = \frac{1}{d_\pi}\int_X\int_X K(x, y) Z_\pi(x, y)^* \, d\mu(x) d\mu(y), \quad \text{for} \quad \pi \in \hat\Gamma,
\]
where the matrices are integrated entrywise. The \emph{inverse Fourier transform} reads
\[
K(x,y) = \sum_{\pi \in \hat\Gamma} \sum_{i,i'=1}^{m_\pi} \hat K(\pi)_{i,i'} Z_\pi(x,y)_{i,i'},
\]
where the series in general converges in $L^2$. The kernel $K$ is positive definite if and only if $\hat K(\pi)$ is positive semidefinite for all $\pi \in \hat\Gamma$ (see Lemma~\ref{lemma:positve definite}). In the special case where the action of $\Gamma$ on $X$ has finitely many orbits and $K$ is positive definite, the above series converges absolutely-uniformly (this is an extension of Bochner's theorem \cite{Bochner1941} to the case of finitely many orbits \cite{Laat2016}). We are interested in the situation of infinitely many orbits, where the above series in general does not converge uniformly.

For each $\pi \in \hat\Gamma$, let $R_{\pi,0} \subseteq R_{\pi,1} \subseteq \ldots$ be finite subsets of $[m_\pi]$ such that $\bigcup_{d=0}^\infty R_{\pi,d} = [m_\pi]$ and such that for each $d$, the set $R_{\pi,d}$ is empty for all but finitely many $\pi$. Let $Z_{\pi,d}$ be the finite principal submatrix of the zonal matrix $Z_\pi$ containing only the rows and columns indexed by elements from $R_{\pi,d}$. Let $C_{\pi,d}$ be the cone of kernels of the form $(x,y) \mapsto \langle A, Z_{\pi,d}(x,y)^* \rangle$, where $A$ ranges over the Hermitian positive semidefinite matrices of appropriate size, and where $\langle A, B \rangle = \mathrm{trace}(AB^*)$ is the trace inner product. Define the $d$th inner approximating cone $C_d$ by the (Minkowski) sum $\sum_{\pi \in \hat\Gamma} C_{\pi,d}$. Then we have 
\[
C_0 \subseteq C_1 \subseteq \ldots \subseteq \Ccal(X \times X; \C)_{\succeq 0}^\Gamma,
\] 
and in Theorem~\ref{thm:energy:uniformly dense union} we show $\bigcup_{d=0}^\infty C_d$ is uniformly dense in $\Ccal(X \times X; \C)_{\succeq 0}^\Gamma$.  

Let 
\[
D_d = \big\{ (K+\overline{K})/2 : K \in C_d \big\} \subseteq C_d,
\]
so that $\bigcup_{d=0}^\infty D_d$ is uniformly dense in $\Ccal(X \times X)_{\succeq 0}^\Gamma$. If all irreducible representations of $\Gamma$ are of real type; that is, if each representation $\pi \in \smash{\hat\Gamma}$ is unitarily equivalent to a representation $\Gamma \to O(d_\pi) \subseteq U(d_\pi)$, then there exists a symmetry adapted system of $X$ consisting of real-valued functions. This means we can choose the symmetry adapted system in such a way that the zonal matrices are real-valued, and
\[
D_d = \Big\{ \sum_{\pi \in \hat\Gamma} \Big\langle F_\pi, Z_{\pi,d}(\cdot, \cdot)^* \Big\rangle: F_\pi \in S_{\succeq 0}^{R_{\pi,d}} \text{ for } \pi \in \hat\Gamma\Big\}.
\] 
If $\smash{\hat \Gamma}$ also contains representations of complex or quaternionic type (these are the two remaining possibilities), then one should construct a \emph{real} symmetry adapted system, where $\pi$ ranges over the real irreducible representations. See \cite{GatermannParrilo2004} or \cite{Serre1977} where this is discussed for finite groups. The irreducible representations of the groups considered in this paper are all of real type.

\subsection{Harmonic analysis on subset spaces}
\label{sect:harmonci analysis on subset spaces}

In this section we show how to construct the sequence $\{D_d\}$ as defined above for the special case where $X = I_t$. We give a construction in two steps: The main step is that we show how to construct a symmetry adapted system for $X_t$, where $X_t$ is a structurally simpler space that contains $I_t$ as an embedding. This yields a sequence of inner approximating cones of $\smash{\Ccal(X_t \times X_t)_{\succeq 0}^\Gamma}$. Then we restrict the domains of the kernels in these inner approximations to the smaller space $I_t \times I_t$ to obtain inner approximations of $\smash{\Ccal(I_t \times I_t)_{\succeq 0}^\Gamma}$.

Let
\[
X_t = \bigcup_{i=0}^t V^i / S_i, 
\]
where $S_i$ is the symmetric group on $i$ elements. The set $X_t$ obtains a topology by using the topology of $V$ and the product, quotient, and disjoint union topologies. We define a continuous action of  $\Gamma$ on $X_t$ by 
\[
\gamma \big\{ (x_{\sigma(1)},\ldots,x_{\sigma(i)}) : \sigma \in S_i\big\} = \big\{ ( \gamma x_{\sigma(1)},\ldots, \gamma x_{\sigma(i)}) : \sigma \in S_i\big\}.
\]
The space $I_t$ embeds as a closed, $\Gamma$-invariant subspace into $X_t$ by the map that sends $\{x_1,\ldots,x_i\} \in I_{=i}$ to $\{ (x_{\sigma(1)},\ldots,x_{\sigma(i)}) : \sigma \in S_i\}$. Notice that we could also embed $I_t$ in $V^t / S_t \cup \{e\}$ (where the empty set maps to an additional point $e$), but for computational reasons it is preferable to embed $I_t$ into $X_t$.

We first show that each kernel in the cone $\smash{\Ccal(I_t \times I_t)_{\succeq 0}^\Gamma}$ is the restriction to $I_t \times I_t$ of a kernel from $\Ccal(X_t \times X_t)_{\succeq 0}^\Gamma$, which shows that if a sequence of inner approximations of $\smash{\Ccal(X_t \times X_t)_{\succeq 0}^\Gamma}$ has dense union, then the corresponding sequence of inner approximations of $\smash{\Ccal(I_t \times I_t)_{\succeq 0}^\Gamma}$ also has dense union.

\begin{lemma}
\label{lem:energy:restircted kernels}
Let $(X,d)$ be a compact metric space with a continuous action of a compact group $\Gamma$, and let $Y$ be a closed $\Gamma$-invariant subspace of $X$. Each kernel $K \in \smash{\Ccal(Y \times Y)_{\succeq 0}^\Gamma}$ is the restriction to $Y \times Y$ of a kernel in $\smash{\Ccal(X \times X)_{\succeq 0}^\Gamma}$.
\end{lemma}
\begin{proof}
Let $K \in \Ccal(Y \times Y)_{\succeq 0}^\Gamma$. By Mercer's theorem there exists a sequence of functions $\{e_i\}$ in $\Ccal(Y)$ such that $\smash{\sum_{i=1}^\infty e_i \otimes \overline{e_i}}$ converges absolutely and uniformly to $K$. In particular this means $\sum_{i=1}^\infty |e_i|^2$ converges uniformly.

Given a point $x \in X$, define
\[
Y_x = \big\{ y \in Y : d(y,x) \leq d(z,x) \text{ for all } z \in X\big\}.
\]
For each $i$ we define the function $c_i \colon X \to \R$ by 
\[
c_i(x) = \min_{y \in Y_x} |e_i(y)|^2 \quad \text{for} \quad x \in X.
\]
We first show $c_i$ is lower semicontinuous. For $\kappa > 0$, let 
\[
Y_{x,\kappa} = Y \setminus \bigcup_{y \in Y_x} B_\kappa(y)^\circ,
\]
where $B_\kappa(y)^\circ$ is the open ball of radius $\kappa$ about $y$. The set $Y_{x,\kappa}$ is closed, so there exist a $\delta > 0$ such that $Y_z \cap Y_{x,\kappa} = \emptyset$ for all $z \in B_\delta(x)$. This means that for every $\varepsilon > 0$, there is a $\delta > 0$ such that $c_i(z) \geq c_i(x) - \varepsilon$ for all $z \in B_\delta(x)$.

Since $X$ is a compact metric space, it is perfectly normal, which implies the existence of a function $\iota \in \Ccal(X)$ such that $\iota|_Y = 1$ and $\iota|_{X \setminus Y} < 1$. By Tietze' extension theorem there exist functions $f_i \in \Ccal(X; K)$ with $f_i|_Y = e_i$. Let
\[
Q_i = \{x \in X : |f_i(x)|^2 \geq c_i(x) + 2^{-i} \}.
\]
The set $Q_i$ is disjoint from $Y$, and it follows from $c_i$ being lower semicontinuous that $Q_i$ is closed and hence compact. So, $M_i = \max_{x \in Q_i} \iota(x)$ exists and is strictly smaller than $1$. Let 
\[
B_i = \max_{x \in Q_i} \frac{|f_i(x)|}{\sqrt{c_i(x) + 1/2^i}},
\]
and let $k_i$ be an integer such that $M_i^{k_i} B_i \leq 1$. It follows that
\[
|g_i|^2 \leq c_i + \frac{1}{2^i}, \quad \text{where} \quad g_i = \iota^{k_i} f_i.
\]

Let $\varepsilon > 0$. Choose $N_1 \in \N$ such that $\sum_{i=N_1}^\infty 1/2^i \leq \varepsilon/2$. Choose $N_2 \in \N$ such that 
\[
\left\|\sum_{i = m}^n |e_i|^2 \right\|_\infty \leq \frac{\varepsilon}{2}
\]
for all $n \geq m \geq N_2$. This is possible because  $\smash{\sum_{i=1}^\infty |e_i|^2}$ converges uniformly and hence Cauchy uniformly. Let $N = \max\{N_1, N_2\}$. Then,
\[
\left\|\sum_{i = m}^n g_i \otimes \overline{g_i}\right\|_\infty \leq \left\|\sum_{i = m}^n |g_i|^2\right\|_\infty \leq \left\|\sum_{i = m}^n c_i \right\|_\infty + \sum_{i = m}^n 1/2^i.
\]
We have 
\[
\left\|\sum_{i = m}^n c_i\right\|_\infty = \sup_{x \in X}\sum_{i = m}^n \min_{y \in Y_x} |e_i(y)|^2.
\]
We use the axiom of choice to select an element $y_x \in Y_x$ for each $x \in X$. Then,
\[
\sup_{x \in X}\sum_{i = m}^n \min_{y \in Y_x} |e_i(y)|^2 \leq \sup_{x \in X}\sum_{i = m}^n |e_i(y_x)|^2 = \sup_{x \in Y}\sum_{i = m}^n |e_i(x)|^2 = \left\|\sum_{i = m}^n |e_i|^2 \right\|_\infty \leq \frac{\varepsilon}{2}.
\]
and $\sum_{i = m}^n 1/2^i \leq \varepsilon/2$, so $\sum_{i = m}^n g_i \otimes \overline{g_i}$ converges uniformly Cauchy and hence uniformly. Let $P$ be the limit function.

Define $K' \in \smash{\Ccal(X \times X)_{\succeq 0}^\Gamma}$ by $K'(x,y) = \int P(\gamma x, \gamma y) \, d\gamma$, where we integrate over the normalized Haar measure of $\Gamma$. Since $P|_{Y \times Y} = K$ is $\Gamma$-invariant, the restriction of $K'$ to $Y \times Y$ equals $K$, which completes the proof.
\end{proof}

To give an explicit construction of a symmetry adapted system for $X_t$ we use symmetric tensor powers. Given a vector space $\Vcal$, denote the $n$th tensor power of $\Vcal$ by $\Vcal^{\otimes n}$; that is, $\Vcal^{\otimes n} = \Vcal \otimes \cdots \otimes \Vcal$ ($n$ times). Given $v_1,\ldots,v_n \in \Vcal$ and $\sigma \in S_n$, let
\[
(\otimes_{i=1}^n v_i)^\sigma = \otimes_{i=1}^n v_{\sigma(i)},
\]
and extend this operation to $\Vcal^{\otimes n}$ by linearity. Define the $n$th \emph{symmetric tensor power} of $\Vcal$ as
\[
\Vcal^{\odot n} = \Big\{ \sum_{\sigma \in S_n} w^\sigma : w \in \Vcal^{\otimes n}\Big\}.
\]
We have $w^\sigma = w$ for all $w \in \Vcal^{\odot n}$ and $\sigma \in S_n$, and $\Vcal^{\odot n} = \mathrm{span} \{ v^{\otimes n} : v \in \Vcal\}$ \cite{ComonGolubLimMourrain2008}. If $\Hcal_1$ and $\Hcal_2$ are Hilbert spaces, then we equip the tensor product $\Hcal_1 \otimes \Hcal_2$ with the inner product $\langle u_1 \otimes u_2, v_1 \otimes v_2 \rangle = \langle u_1, v_1 \rangle \langle u_2, v_2 \rangle$, where we extend linearly in the first and antilinearly in the second component. We denote the Hilbert space obtained by taking the completion in the metric given by this inner product by $\Hcal_1 \mathbin{\hat\otimes} \Hcal_2$.  A symmetric tensor power $\smash{\Hcal^{\odot n}}$ of a Hilbert space $\Hcal$ gets a metric as a subspace of $\smash{\Hcal^{\hat \otimes n}}$, and we denote the completion in this metric by $\Hcal^{\hat\odot n}$. We have 
\[
\Hcal^{\hat\odot n} = \mathrm{closure}(\mathrm{span}(\{ v^{\otimes n} : v \in \Hcal\})),
\]
where the closure is in $\Hcal^{\hat\otimes n}$.
The \emph{(inner) tensor product representation} 
\[
\pi_1 \otimes \pi_2 \colon \Gamma \to U(\Hcal_1 \mathbin{\hat\otimes} \Hcal_2)
\]
of two unitary representations $\pi_1 \colon \Gamma \to U(\Hcal_1)$ and $\pi_2 \colon \Gamma \to U(\Hcal_2)$ is defined by
\[
(\pi_1 \otimes \pi_2)(\gamma) (v_1 \otimes v_2) = (\pi(\gamma) v_1) \otimes (\pi(\gamma) v_2),
\]
and this definition extends to finite products and finite (symmetric) powers.

Let $\mu$ be a strictly positive $\Gamma$-invariant Radon measure on $X$. This defines a strictly positive, $\Gamma$-invariant Radon measure $\nu$ on $X_t$ by
\[  
\nu(f) = \sum_{i=0}^t \idotsint_V f(\{ (x_{\sigma(1)},\ldots,x_{\sigma(i)}) : \sigma \in S_i\big\}) \, d\mu(x_1) \cdots d\mu(x_i).
\]
For each $0 \leq i \leq t$, we define the operator 
\[
L_i \colon L^2(V, \mu; \C)^{\hat \odot i} \to L^2(X_t, \nu; \C)
\]
by setting
\[
L_i(f^{\otimes i})(\{ (x_{\sigma(1)},\ldots,x_{\sigma(i)}) : \sigma \in S_i\}) = \prod_{k=1}^i f(x_k)
\]
for $f \in L^2(V, \mu; \C)$ and extending by linearity and continuity. These are isometric, $\Gamma$-equivariant operators with pairwise orthogonal images, such that
\[
\bigoplus_{i=0}^t L_i(L^2(V, \mu; \C)^{\hat \odot i}) = L^2(X_t, \nu; \C),
\]
and
\[
\sum_{i=0}^t L_i(\Ccal(V; \C)^{\odot i})
\]
is uniformly dense in $\Ccal(X_t; \C)$. 

We assume we have a symmetry adapted system of $V$. As discussed in the previous section, such a system defines a sequence $\smash{\{\Hcal_k\}_{k=1}^m}$ (where $1 \leq m \leq \infty$) of pairwise orthogonal, $\Gamma$-irreducible (and hence finite dimensional) subrepresentations of $L^2(V, \mu; \C)$, such that the algebraic sum $\sum_{k=1}^m \Hcal_k$ is uniformly dense in $\Ccal(V; \C)$.

The spaces 
\[
\bigotimes_{k=1}^m \Hcal_k^{\odot \tau_k}, \quad \text{for} \quad \tau \in D_i = \Big\{ \tau \in \N_0^m : \sum_{k=1}^m \tau_k = i\Big\} \quad \text{and} \quad 0 \leq i \leq t,
\]
are pairwise orthogonal, $\Gamma$-invariant subspaces of $L^2(V,\mu; \C)^{\otimes i}$. Notice that the above tensor products are finite even if $m = \infty$, since $\tau_k$ is nonzero for at most finitely many $k$. For each $0 \leq i \leq t$, we will define (see below) a $\Gamma$-equivariant, unitary operator
\[
T_i \colon \underset{\tau \in D_i}{\widehat\bigoplus} \; \bigotimes_{k=1}^m \Hcal_k^{\odot\tau_k} \to L^2(V,\mu; \C)^{\hat\odot i},
\]
so that
\[
\sum_{\tau \in D_i} T_i\Big(\bigotimes_{k=1}^m \Hcal_k^{\odot \tau_k}\Big)
\]
is uniformly dense in $\Ccal(V; \C)^{\odot i}$.

The finite dimensional spaces $\otimes_{k=1}^m \Hcal_k^{\odot \tau_k}$, for $\tau \in D_i$ and $0 \leq i \leq t$, decompose into $\Gamma$-irreducible representations; that is, there exist $\Gamma$-equivariant, unitary operators
\[
M_\tau \colon \underset{\pi \in R_\tau}{\bigoplus} \Hcal_\pi \to \bigotimes_{k=1}^m \Hcal_k^{\odot \tau_k},
\]
where $\Hcal_\pi$ is the Hilbert space of the irreducible representation $\pi \in R_\tau$, and where $R_\tau$ is some finite subset of $\smash{\hat\Gamma}$. In Section~\ref{subsect:explicit harmonic analysis} we construct the sets $R_\tau$ and the operators $M_\tau$ explicitly for the case where $V = S^2$, $\Gamma = O(3)$, and $t = 2$.

By composing the operators defined above we can define a symmetry adapted system of $X_t$. For each $\pi \in \hat\Gamma$, let $\{e_{\pi,1},\ldots,e_{\pi,d_\pi}\}$ be an orthonormal basis of $\Hcal_\pi$. Then,
\begin{equation}
\label{eq:symmetry adapted system}
\Big\{L_i(T_i(M_\tau(e_{\pi,j}))) : 0 \leq i \leq t, \, \tau \in D_i, \, \pi \in R_\tau, \, j \in [d_\pi] \Big\}
\end{equation}
is a symmetry adapted system of $X_t$.

In the remainder of this section we give the precise definition of $T_i$ and show it is a well-defined, $\Gamma$-equivariant, unitary operator. This generalizes a result from \cite{AnsemilFloret1998} from finite to infinite direct sums and from vector spaces to representations. Given $\tau \in D_i$, let $A_\tau$ be the subgroup consisting of all $\sigma \in S_i$ for which the set
\[
\Big\{\sum_{k=1}^{j-1} \tau_k+1,\sum_{k=1}^{j-1} \tau_k+2, \ldots, \sum_{k=1}^{j} \tau_k\Big\}
\]
is invariant under the permutation $\sigma$ on $[i]$ for each $j \in [m]$. Let $B_\tau$ be the left coset space of $S_i$ modulo $A_\tau$. Given $\tau \in D_i$, $\smash{w \in \otimes_{k=1}^m \Hcal_k^{\odot \tau_k}}$, and $[\sigma] \in B_\tau$, the operation $w \mapsto w^\sigma$ is well-defined, and
\[
\frac{1}{|B_\tau|} \sum_{[\sigma] \in B_\tau} w^\sigma \in L^2(V, \mu; \C)^{\odot i}.
\]
Moreover, if we fix an element $w_\tau \in \smash{\otimes_{k=1}^m \Hcal_k^{\odot \tau_k}}$ for every $\tau \in D_i$, then $w_\tau^\sigma$ and $w_{\tau'}^{\sigma'}$ are orthogonal whenever $\tau \neq \tau'$ or $\sigma \neq \sigma'$. Hence, we can define $T_i$ by setting
\[
T_i(w) = \frac{1}{|B_\tau|} \sum_{[\sigma] \in B_\tau} w^\sigma, \quad \text{for} \quad w \in \bigotimes_{k=1}^m \Hcal_k^{\odot \tau_k} \quad \text{and} \quad \tau \in D_i,
\]
and extending by linearity and continuity.

\begin{lemma}
For each $0 \leq i \leq t$, $T_i$ is a $\Gamma$-equivariant, unitary operator, and 
\[
\sum_{\tau \in D_i} T_i\Big(\bigotimes_{k=1}^m \Hcal_k^{\odot \tau_k}\Big)
\]
is uniformly dense in $\Ccal(V; \C)^{\odot i}$.
\end{lemma}
\begin{proof}
Given $w \in \otimes_{k=1}^m \Hcal_k^{\odot \tau_k}$, we have
\[
\|T_i(w)\| = \left\|\frac{1}{|B_d|} \sum_{[\sigma] \in B_\tau} w^\sigma\right\| = \frac{1}{|B_\tau|} \sum_{[\sigma] \in B_\tau} \|w^\sigma\| = \|w\|,
\]
so $T_i$ is a linear isometry.

The span of the elements of the form $(\sum_{k = 1}^m v_k)^{\otimes i}$, where $v_k \in \Hcal_k$ for  $k \in [m]$ and $v_k = 0$ for all but finitely many $k$, is uniformly dense in $\Ccal(V; \C)^{\odot i}$. Such an element has a preimage under $T_i$:
\begin{align*}
\Big(\sum_{k = 1}^m v_k\Big)^{\otimes i}
&= \sum_{k_1,\ldots,k_i = 1}^m\, \bigotimes_{j=1}^i v_{k_j} = \sum_{\{k_1,\ldots,k_i\} \subseteq [m]}\, \sum_{\sigma \in S_i} \, \bigotimes_{j=1}^i v_{k_{\sigma(j)}}\\
&= \sum_{\tau \in D_i} \, \sum_{[\sigma] \in B_\tau} \Big(\bigotimes_{k=1}^m v_k^{\otimes \tau_k}\Big)^\sigma = T_i\Big(\sum_{\tau \in D_i} \bigotimes_{k=1}^m v_k^{\otimes \tau_k}\Big).
\end{align*}
So, 
\[
\sum_{\tau \in D_i} T_i\Big(\bigotimes_{k=1}^m \Hcal_k^{\odot \tau_k}\Big)
\]
is uniformly dense in $\Ccal(V; \C)^{\odot i}$. This means that $T_i$ is an isometry whose image is dense in $L^2(V, \mu; \C)^{\odot i}$, and $T_i$ therefore is a unitary operator.

Since the spaces $\Hcal_k$ are $\Gamma$-invariant, the spaces $\otimes_{k=1}^m \Hcal_k^{\odot \tau_k}$, for $\tau \in D_i$, are also $\Gamma$-invariant. So, for $w \in \otimes_{k=1}^m \Hcal_k^{\odot \tau_k}$, with $\tau \in D_i$, we have
\begin{align*}
T_i(\pi^{\otimes i}(\gamma) w) 
&= \frac{1}{|B_\tau|} \sum_{[\sigma] \in B_\tau} (\pi^{\otimes i}(\gamma) w)^\sigma = \frac{1}{|B_\tau|} \sum_{[\sigma] \in B_\tau} \pi^{\otimes i}(\gamma) w^\sigma\\
&= \pi^{\otimes i}(\gamma)  \left(\frac{1}{|B_\tau|} \sum_{[\sigma] \in B_\tau} w^\sigma\right) = \pi^{\otimes i}(\gamma) T_i(w),
\end{align*}
which means that $T_i$ is $\Gamma$-equivariant.
\end{proof}

\subsection{Explicit computations for the sphere}
\label{subsect:explicit harmonic analysis}

In this section we explicitly construct the sequence $\{D_d\}$ of inner approximations of $\smash{\Ccal(I_t \times I_t)_{\succeq 0}^\Gamma}$ (see Section~\ref{subsecSymmetry adapted systems and zonal matrices}), where $V = S^2$, $\Gamma = O(3)$, and $t = 2$. As explained in Section~\ref{sect:harmonci analysis on subset spaces}, for this we need to construct a symmetry adapted system for
\[
X_2 = \bigcup_{i=0}^2 V^i / S_i.
\]

Let $\mathcal H_\ell \subseteq \Ccal(S^2; \C)$ be the space of spherical harmonics of degree $\ell$. A \emph{spherical harmonic} is the restriction to $S^2$ of a homogeneous polynomial in $\C[x,y,z]$ that vanishes under the Laplacian $\Delta = \partial^2/\partial x^2 + \partial^2/\partial y^2 + \partial^2/\partial z^2$; see for instance \cite{Gelfand1958}. The space $\mathcal H_\ell$ has dimension $2\ell+1$. Moreover, these spaces are orthogonal and form irreducible subrepresentations of the unitary representation 
\[
L_s \colon SO(3) \to U(L^2(S^2, \mu; \C)), \, L_s(\gamma)f(x) = f(\gamma^{-1} x),
\]
where $\mu$ is the invariant measure on the sphere with normalization $\mu(S^2) = 1$.
The subrepresentations $\Hcal_\ell$ in fact form a complete set of subrepresentations: The algebraic sum of the spaces $\Hcal_\ell$ is uniformly dense in $\Ccal(S^2; \C)$, and $L^2(S^2, \mu; \C)$ decomposes as the Hilbert space direct sum of the spaces $\Hcal_\ell$. Moreover, every irreducible unitary representation $SO(3)$ is equivalent to $\Hcal_\ell$ for some $\ell$.

An explicit orthonormal  basis of $\Hcal_\ell$ is given by the \emph{Laplace spherical harmonics} $\smash{Y_\ell^m}$ for $-\ell \leq m \leq \ell$. These functions are usually defined using spherical coordinates, such as 
$
Y_\ell^m(\vartheta, \varphi) = (-1)^m c_\ell^m P_\ell^m(\cos(\varphi)) e^{i m \vartheta},
$
where we use the Condon--Shortley phase convention, and the spherical parameterization
\[
x = \cos(\vartheta) \sin(\varphi), \quad y = \sin(\vartheta) \sin(\varphi), \quad z = \cos(\varphi).
\]
Here 
\[
c_\ell^m = \sqrt{(2\ell+1)\frac{(\ell-m)!}{(\ell+m)!}}
\]
is a normalization constant, and $P_\ell^m$ is the $\ell$th \emph{associated Legendre polynomial} of order $m$. For $m \geq 0$ the polynomials $P_\ell^m$ and $P_\ell^{-m}$ are defined as
\[
P_\ell^m(z) = (1-z^2)^{m/2} \frac{d^m}{dz^m}\left(P_\ell(z)\right), \quad P_\ell^{-m}(z) = (-1)^m \frac{(\ell-m)!}{(\ell+m)!}P_\ell^m,
\]
where
$
P_\ell(z) = (2^\ell \ell!)^{-1} \frac{d^\ell}{dz^\ell} (z^2 -1)^\ell
$
is the $\ell$th Legendre polynomial.

In cartesian coordinates, $Y_\ell^m$ has the form
\[
(-1)^m c_\ell^m P_l^m\left(\frac{z}{\sqrt{x^2+y^2+z^2}}\right) \left(\frac{x+\mathrm{sign}(m)iy}{\sqrt{x^2+y^2}}\right)^m.
\]
By using the identity $x^2+y^2+z^2 = 1$ and the above definition of the associated Legendre polynomials, for $m \geq 0$ we get
\[
Y_\ell^m(x,y,z) = (-1)^m c_\ell^m (x+iy)^m \frac{d^mP_\ell(z)}{dz^m},\, Y_\ell^{-m}(x,y,z) = c_\ell^m (x-iy)^m \frac{d^mP_\ell(z)}{dz^m}.
\]
From the definition of $P_\ell$ we see that if $\ell$ is even (odd), every term of $P_\ell$ has even (odd) degree, so we can multiply the terms in $Y_\ell^m(x,y,z)$ with appropriate powers of $x^2 + y^2 + z^2$ to make $Y_\ell^m(x,y,z)$ into a homogeneous polynomial of degree $\ell$.

In general, an inner tensor product (see previous section) of irreducible representations is not irreducible. By the above discussion we know that a tensor product $\Hcal_{\ell_1} \otimes \Hcal_{\ell_2}$ must be isomorphic to a direct sum of the spaces $\Hcal_\ell$. Indeed, we have 
\[
\Hcal_{\ell_1} \otimes \Hcal_{\ell_2} \simeq \Hcal_{|\ell_1-\ell_2|} \oplus \cdots \oplus \Hcal_{\ell_1+\ell_2},
\]
where the $SO(3)$-equivariant, unitary operator 
\[
\Phi_{\ell_1,\ell_2} \colon \Hcal_{\ell_1} \otimes \Hcal_{\ell_2} \to \Hcal_{|\ell_1-\ell_2|} \oplus \cdots \oplus \Hcal_{\ell_1+\ell_2}
\]
is rather nontrivial, but can be given explicitly by using the Clebsch--Gordan coefficients \cite{Gelfand1958}. For this we set
\[
\Phi_{\ell_1,\ell_2}(Y_{\ell_1}^{m_1} \otimes Y_{\ell_2}^{m_2}) = \sum_{\ell = |\ell_1-\ell_2|}^{\ell_1+\ell_2} \sum_{m=-\ell}^\ell C_{\ell_1,m_1,\ell_2,m_2}^{\ell, m} Y_\ell^m
\]
and extend by linearity, where the \emph{Clebsch--Gordan coefficients} are given by
 \begin{align*}
 & C_{\ell_1,m_1,\ell_2,m_2}^{\ell, m} = \delta_{m_1+m_2=m} \Big(  \frac{(2\ell+1) (\ell_1+\ell_2-\ell)! (\ell_1-\ell_2+\ell)! (-\ell_1+\ell_2-\ell)!}{(\ell_1+\ell_2+\ell+1)!}\\[-0.4em]
 &\quad\quad\quad  \cdot (\ell_1+m_1)! (\ell_1-m_1)!(\ell_2+m_2)!(\ell_2-m_2)!(\ell+m)!(\ell-m)!\Big)^{1/2} \\[-0.4em]
 &\quad\quad\quad \cdot \sum_{\nu=-\infty}^\infty (-1)^\nu \Big( \nu! (\ell_1+\ell_2-\ell-\nu)!(\ell_1-m_1-\nu)!(\ell_2+m_2-\nu)! \\[-0.8em]
 &\quad \quad\quad\quad\quad\quad \cdot (\ell-\ell_2+m_1+\nu)! (\ell-\ell_1-m_2+\nu)! \Big)^{-1}.
 \end{align*}

Since the Clebsch--Gordan coefficients are real numbers, it follows that
\[
\Phi_{\ell_1,\ell_2}^{-1}(Y_\ell^m) = \sum_{m_1 = -\ell_1}^{\ell_1} \sum_{m_2 = -\ell_2}^{\ell_2} C_{\ell_1,m_1,\ell_2,m_2}^{\ell, m} Y_{\ell_1}^{m_1} \otimes Y_{\ell_2}^{m_2}.
\]
For any set of scalars $\{c_m\}$ we have
\begin{align*}
\Phi_{\ell',\ell'}\Big(\sum_{m_1 = -\ell'}^{\ell'} & c_{m_1} Y_{\ell'}^{m_1} \otimes \sum_{m_1 = -\ell'}^{\ell'} c_{m_1} Y_{\ell_1}^{m_1} \Big)\\
&= \sum_{\ell = 0}^{2\ell'} \sum_{m=-\ell}^\ell \left(\sum_{m_1, m_2 = -\ell'}^{\ell'} c_{m_1} c_{m_2} C_{\ell',m_1,\ell',m_2}^{\ell, m}\right) Y_\ell^m,
\end{align*}
and by using the symmetry relation 
\[
C_{\ell_2,m_2,\ell_1,m_1}^{\ell, m} = (-1)^{\ell_1+\ell_2-\ell} C_{\ell_1,m_1,\ell_2,m_2}^{\ell, m}
\]
we obtain 
\[
\sum_{m_1, m_2 = -\ell'}^{\ell'} c_{m_1} c_{m_2} C_{\ell',m_1,\ell',m_2}^{\ell, m} = 0\]
for all odd numbers $\ell$. This shows that
\[
\Phi_{\ell',\ell'}(\Hcal_{\ell'}^{\odot 2}) \subseteq \Hcal_0 \oplus \Hcal_2 \oplus \cdots \oplus \Hcal_{2\ell'}.
\]
We have 
\begin{align*}
\dim(\Hcal_{\ell'}^{\odot 2}) &= \binom{\dim(\Hcal_{\ell'}) + 1}{2} = \binom{2\ell' + 2}{2} = 2(\ell')^2 + 3\ell' + 1  \\
&= \sum_{k=0}^{\ell'} (4k+1) =  \dim(\Hcal_0 \oplus \Hcal_2 \oplus \cdots \oplus \Hcal_{2\ell'}),
\end{align*}
and therefore
\[
\Hcal_{\ell'}^{\odot 2} \simeq \Hcal_0 \oplus \Hcal_2 \oplus \cdots \oplus \Hcal_{2\ell'}.
\]
Let $\Phi_\ell$ be the isomorphism $\Hcal_\ell^{\odot 2} \to \Hcal_0 \oplus \Hcal_2 \oplus \cdots \oplus \Hcal_{2\ell}$ defined by $\Phi_\ell = \Phi_{\ell,\ell}|_{\Hcal_\ell^{\odot 2}}$.

We have seen how $L^2(S^2, \mu; \C)$ decomposes into $SO(3)$-irreducible representations, and how tensor products and symmetric tensor powers of these irreducibles decompose into irreducibles. In the next section it will be essential that instead of the group $SO(3)$, we consider the full symmetry group $O(3)$ of $S^2$. 
The special orthogonal group $SO(3)$ forms a normal subgroup of $O(3)$. Since $\R^3$ is odd dimensional, the inversion operation $x \mapsto -x$ is not contained in $SO(3)$. This operation, which we denote by $-I$, generates a $2$ element normal subgroup of $O(3)$, and the orthogonal group $O(3)$ is isomorphic to the direct product $\Z_2 \times SO(3)$. Thus, for each irreducible representation $\Hcal_\ell$ of $SO(3)$, we define two nonequivalent irreducible representations $\smash{\pi_\ell^p} \colon O(3) \to U(\Hcal_\ell^p)$, with $p = \pm 1$, where $\smash{\Hcal_\ell^{+1}}$ and $\smash{\Hcal_\ell^{-1}}$ are both isomorphic to $\Hcal_\ell$ as Hilbert spaces, and where $\smash{\pi_\ell^p|_{SO(3)}}$ is equivalent to $\Hcal_\ell$, but where $\pi_\ell^p(-I) f = pf$ for all $f$. It follows that $\pi_\ell^p$ is a subrepresentation of the unitary representation
\[
L \colon O(3) \to U(L^2(S^2, \mu; \C)), \, L(\gamma)f(x) = f(\gamma^{-1} x)
\]
if and only if $p = (-1)^\ell$. For $f_1 \in \Hcal_{\ell_1}^{p_1}$ and $f_2 \in \Hcal_{\ell_2}^{p_2}$ we have
\[
(\pi_{\ell_1}^{p_1} \otimes \pi_{\ell_2}^{p_2})(-I)(f_1 \otimes f_2) = \pi_{\ell_1}^{p_1}(-I)f_1 \otimes \pi_{\ell_2}^{p_2}(-I)f_2 = p_1p_2 (f_1 \otimes f_2),
\]
which implies
\[
\Hcal_{\ell_1}^{p_1} \otimes \Hcal_{\ell_2}^{p_2} \simeq \Hcal_{|\ell_1-\ell_2|}^{p_1p_2} \oplus \cdots \oplus \Hcal_{\ell_1+\ell_2}^{p_1p_2} \quad \text{and} \quad (\Hcal_{\ell'}^p)^{\odot 2} \simeq \Hcal_0^{+1} \oplus \Hcal_2^{+1} \oplus \cdots \oplus \Hcal_{2\ell'}^{+1}.
\]
The operators $\Phi_{\ell_1,\ell_2}$ and $\Phi_\ell$ defined above now become $O(3)$-equivariant, unitary operators 
\[
\Phi_{\ell_1,\ell_2} \colon \Hcal_{\ell_1}^{p_1} \otimes \Hcal_{\ell_2}^{p_2} \to \Hcal_{|\ell_1-\ell_2|}^{p_1p_2} \oplus \cdots \oplus \Hcal_{\ell_1+\ell_2}^{p_1p_2}
\]
and
\[
\Phi_{\ell} \colon (\Hcal_{\ell'}^p)^{\odot 2} \to \Hcal_0^{+1} \oplus \Hcal_2^{+1} \oplus \cdots \oplus \Hcal_{2\ell'}^{+1}.
\]

We use the operators $\Phi_{\ell_1,\ell_2}$ and $\Phi_\ell$ to give an explicit definition of the operators $M_\tau$ from Section~\ref{sect:harmonci analysis on subset spaces}. Let $e_\ell$ denote the vector in $\N_0^\infty$ with $(e_\ell)_{\ell'} = 1$ if $\ell = \ell'$ and $0$ otherwise. For $\tau \in D_0$, $M_\tau$ becomes the identity operator $\Hcal_0^1 \to \C$. Each $\tau \in D_1$ is of the form $\tau = e_\ell$ for some $\ell$, and for such a vector $\tau$, the operator $\smash{M_\tau}$ is the identity operator $\Hcal_\ell^p \to \Hcal_\ell^p$, where $p = {(-1)^\ell}$. If $\tau \in D_2$ is of the form $\tau = e_{\ell_1} + e_{\ell_2}$ with $\ell_1 \neq \ell_2$, then $M_\tau$ is given by the operator $\smash{\Phi_{\ell_1,\ell_2}^{-1}}$. If $\tau \in D_2$ is of the form $2e_{\ell}$ for some $\ell \in \N_0$, then $M_\tau$ is given by $\smash{\Phi_\ell^{-1}}$. 

The representations in $\hat\Gamma$ can be indexed by $(\ell, p) \in \N_0 \times \{\pm 1\}$, and a symmetry adapted system of $X_2$ has the form 
\[
\big\{ e_{(\ell, p), \tau, m} : \ell \in \N_0, \, p = \pm 1, \, \tau \in R_{(\ell,p)}, \, -\ell \leq m \leq \ell\big\},
\]
where $R_{(\ell,p)} = R^0_{(\ell,p)} \cup R^1_{(\ell,p)} \cup R^2_{(\ell,p)}$,
\[
R^0_{(\ell, p)} = \begin{cases}
\{0\} & \text{if } \ell = 0 \text{ and } p = 1,\\
\emptyset & \text{otherwise},
\end{cases} \quad 
R^1_{(\ell, p)} = \begin{cases}
\{e_\ell\} & \text{if } p = (-1)^\ell,\\
\emptyset & \text{otherwise},
\end{cases}
\]
and
\[
R^2_{(\ell, p)} = \Big\{e_{\ell_1} + e_{\ell_2} : \delta_{2 \nmid \ell} \leq |\ell_1 - \ell_2| \leq \ell \leq \ell_1 + \ell_2, \, (-1)^{\ell_1+\ell_2} = p\Big\},
\]
where $\delta_{2\nmid\ell}$ is $1$ if $\ell$ is odd, and $0$ if $\ell$ is even.
Here the basis element $e_{(\ell, p), \tau, m}$ can be computed as $L_i(T_i(M_\tau(Y_\ell^m)))$, where $i = \sum_{\ell'} \tau_{\ell'}$.

The rows and columns of the zonal matrices  $Z_{(\ell,p)}(T, T')$ constructed using this symmetry adapted system are indexed by $R_{(\ell, p)}$. To obtain the finite  dimensional inner approximating cones we need to select finite subsets $R_{(\ell,p),d} \subseteq R_{(\ell, p)}$ so that $\smash{\bigcup_{d=0}^\infty} R_{(\ell,p),d} = R_{(\ell, p)}$, and for each $d$ only finitely many sets $R_{(\ell,p),d}$ are nonempty. One particularly nice way (see below) to do this is to set
\[
R_{(\ell,p),d} = \big\{ \tau \in R_{(\ell,p)} : \sum_i \tau_i \leq d \big\}.
\]

All irreducible representations of $O(3)$ are of real type, and even though the above symmetry adapted system is not real valued, the zonal matrices constructed above are all real valued. As shown in Section~\ref{subsecSymmetry adapted systems and zonal matrices}, we can use these to construct the sequence $\{D_d\}$ of inner approximations of $\smash{\Ccal(X_2 \times X_2)_{\succeq 0}^\Gamma}$. By Lemma~\ref{lem:energy:restircted kernels} we then get the desired sequence of inner approximations of $\smash{\Ccal(I_2 \times I_2)_{\succeq 0}^\Gamma}$.

The symmetry adapted system above is constructed in such a way that the functions
\begin{equation}
\label{eq:a2polynomial}
(x_1,\ldots,x_i) \mapsto A_2 Z_{(\ell, p)}(\cdot, \cdot)_{\tau, \tau'} (\{x_1,\ldots,x_i\}),
\end{equation}
for $0 \leq i \leq 4$ and $\tau,\tau' \in R_{(\ell,p),d}$, are polynomials (of degree $2d$ in $3i$ variables), which will be important in the next section. For $n > 3$ and $t > 2$ the construction is more involved, but by using higher dimensional spherical harmonics one can construct a symmetry adapted system in the same way again having the above polynomial property.

\section{Semidefinite programs with semialgebraic constraints}
\label{sec:reduction to sdps}

In this section we reduce the dual problem $E_t^*$ for the Riesz $s$-energy problem on $S^{n-1}$, with $s \in \N$, to a sequence of semidefinite programs with semialgebraic constraints. Here, by a \emph{semialgebraic constraint}, we mean the requirement that a polynomial, whose coefficients depend linearly on the entries of the positive semidefinite matrix variable(s), is nonnegative on a given basic closed semialgebraic set. A \emph{basic closed semialgebraic set} in $\R^n$ is a subset that has a description of the form
\[
S(g_1,\ldots,g_m) =\big\{x \in \R^n : g_i(x) \geq 0 \text{ for } i \in [m]\big\},
\]
where $g_1,\ldots,g_m \in \R[x_1,\ldots,x_n]$. In Section~\ref{sec:putinars theorem for invaraint} we show how these programs can be approximated by semidefinite programs and how symmetries in the semialgebraic constraints can be exploited to (further) block diagonalize the semidefinite programming formulations.

Following Section~\ref{sec:derivation of the hierarchy via relaxations} and Section~\ref{sec:infinite dimensional optimization}, the second step in the dual hierarchy for the Riesz $s$-energy problem on $S^{n-1}$ reads
\begin{align*}
E_t^* = \sup \Big\{ \sum_{i=0}^{2t} \binom{N}{i} a_i : \; & a \in \R^{\{0,\ldots,2t\}}, \, K \in \Ccal(I_t \times I_t)_{\succeq 0},\\[-1em]
& a_i + A_tK(S) \leq f(S) \text{ for } S \in I_{=i} \text{ and } i = 0,\ldots,2t \Big\},
\end{align*}
where $I_t$ is the set of independent sets of cardinality at most $t$ in a topological packing graph $G$ on $S^{n-1}$ as discussed in Section~\ref{sec:derivation of the hierarchy via relaxations}, and where
\[
f(S) = \begin{cases}
\| x - y\|^{-s} & \text{if } S = \{x,y\} \text{ with } x \neq y,\\
0 & \text{otherwise.}
\end{cases}
\]
From the definition of $G$ we obtain a $U \in (-1, 1)$ such that two vertices $x$ and $y$ in $S^{n-1}$ are adjacent if and only if $x \cdot y \geq U$.

In the symmetrized version of this problem, as derived in Section~\ref{sec:energy:symmetry reduction}, we restrict to $O(n)$-invariant kernels. In Section~\ref{subsect:explicit harmonic analysis} we give a sequence $\{D_d\}$ of inner approximating cones 
of $\smash{\Ccal(I_t \times I_t)_{\succeq 0}^\Gamma}$, where $\Gamma = O(n)$.  By replacing $\Ccal(I_t \times I_t)_{\succeq 0}^\Gamma$ with $D_d$, we obtain the following sequence of approximations:
\begin{align*}
E_{t,d}^* = \sup \Big\{ \sum_{i=0}^{2t} \binom{N}{i} a_i : \; & a \in \R^{\{0,\ldots,2t\}}, \, F_\pi \in S_{\succeq 0}^{R_{\pi,d}} \text{ for } \pi \in \hat\Gamma,\\[-1em]
& a_i + A_2K(S) \leq f(S) \text{ for } S \in I_{=i} \text{ and } i = 0,\ldots,2t \Big\},
\end{align*}
where
\[
K(J,J') = \sum_{\pi \in \hat\Gamma}  \Big\langle F_\pi, Z_{\pi,d}(J, J)^*\Big\rangle.
\] 
For each $d$, the problem $E_{t,d}^*$ has finite dimensional variable space, and as shown in Section~\ref{sec:energy:symmetry reduction}, we have $E_{t,d}^* \to E_t^* = E_t$ as $d \to \infty$, and, as shown in Section~\ref{sec:energy convergence to the optimal energy}, we have $E_{N,d}^* \to E$ as $d \to \infty$.

Given $0 \leq i \leq 2t$, let $\R[x_1,\ldots,x_i]$ be the ring of real polynomials in $ni$ variables, where each $x_k$ denotes a vector of $n$ variables. As shown in Section~\ref{subsect:explicit harmonic analysis}, there exist polynomials $p_i \in \R[x_1,\ldots,x_i]$ such that
\[
p_i(x_1,\ldots,x_i) = a_i + A_tK(\{x_1,\ldots,x_i\}) \quad \text{for all} \quad \{x_1,\ldots,x_i\} \in I_{=i},
\]
where the coefficients of $p_i$ depend linearly on the entries of the vector $a$ and the matrices $F_{(\ell,p)}$. For the case where $n=3$ and $t=2$ we have explicitly derived these polynomials in Section~\ref{subsect:explicit harmonic analysis}.

By construction, the polynomials $p_i \in \R[x_1,\ldots,x_i]$ are $O(n)$-invariant:
\[
p_i(x_1,\ldots,x_i) = p_i(\gamma x_1, \ldots, \gamma x_i) \quad \text{for all} \quad x_1,\ldots,x_i \in S^2 \quad \text{and} \quad \gamma \in O(n).
\]
So, by a theorem from invariant theory (see, for instance, \cite{KraftProcesi1996}*{Theorem 10.2}) it follows that $p_i$ can be written as a polynomial in the inner products 
\[
x_1 \cdot x_1, x_1 \cdot x_2, \ldots, x_i \cdot x_i.
\]
On the sphere we have the identity $x_1 \cdot x_1 = \ldots = x_i \cdot x_i = 1$, so there exists a (in general non unique) polynomial $q_i \in \R[u_1,\ldots,u_{\binom{i}{2}}]$ such that
\begin{equation}
\label{eqpq}
p_i(x_1,\ldots,x_i) = q_i(x_1 \cdot x_2, x_1 \cdot x_3, \ldots, x_{i-1} \cdot x_i) \quad \text{for all} \quad x_1,\ldots,x_i \in S^{n-1},
\end{equation}
where the coefficients of $q_i$ again depend linearly on the entries of the vector $a$ and the matrices $F_{(\ell,p)}$. The above result is nonconstructive, and in Section~\ref{sec:enegy:computations} we show how we compute the polynomials $q_i$. The use of this theorem is why we need $O(n)$-invariance instead of just $SO(n)$-invariance, for otherwise the polynomials $q_i$ would also need to depend on the determinants of the $n \times n$ matrices whose columns are given by $n$ distinct vectors from $\{x_1,\ldots,x_i\}$, which means we would have too many variables to be able to perform computations.

The degenerate polynomials $q_0$ and $q_1$ have $0$ variables; they are linear combinations of the entries of the vector $a$ and the matrices $F_{(\ell, p)}$. The constraints $a_0 + A_tK|_{I_{=0}} \leq 0$ and $a_1 + A_tK|_{I_{=1}} \leq 0$ in $E_{t,d}^*$, where we use that $f|_{I_1} \equiv 0$, therefore reduce to the two linear constraints $q_0 \leq 0$ and $q_1 \leq 0$.

For distinct $x,y \in S^{n-1}$ we have 
\[
f(\{x,y\}) = \frac{1}{\|x-y\|_2^s} = \frac{1}{(2-2 x\cdot y)^{s/2}}.
\]
By using the substitution $w = \sqrt{2-2u}$, we can reformulate the constraint 
\[
a_2 + A_tK|_{I_{=2}} \leq f|_{I_{=2}}
\]
in $E_{t,d}^*$ as the semialgebraic constraint
\[
 w^s \,q_2(1- w^2/2) \leq 1 \quad  \text{for} \quad w \in [\sqrt{2-2U}, 2].
\]
If $s$ is even, we can use a more efficient formulation (in terms of the degree of the polynomials), where we write the constraint $a_2 + A_tK|_{I_{=2}} \leq f|_{I_{=2}}$ in $E_{t,d}^*$ as the semialgebraic constraint
\[
(2-2u)^{s/2} q_2(u) \leq 1 \quad  \text{for} \quad u \in [-1, U].
\]
The set of independent sets of cardinality $i$ can be described as
\[
I_{=i} = \Big\{ \{x_1,\ldots,x_i\} \subseteq S^{n-1} : x_k \cdot x_{k'} \leq U \text{ for } 1 \leq k < k' \leq i \Big\}.
\]
So, with
\[
P_i = \big\{ (x_1 \cdot x_2, x_1 \cdot x_3, \ldots, x_{i-1} \cdot x_i) : \{x_1, \ldots, x_i\} \in I_{=i}\big\},
\]
a constraint of the form $a_i + A_tK|_{I_{=i}} \leq  0$, for $i \in \{3,4,\ldots,2t\}$, can be written as $q_i|_{P_i} \leq 0$. 
This means we can write the problem $E_{t,d}^*$ as
\begin{align}
\label{eq:etd}
E_{t,d}^* = \sup \Big\{ \sum_{i=0}^{2t} \binom{N}{i} a_i : \; & a \in \R^{\{0,\ldots,{2t}\}}, \, F_\pi \in S_{\succeq 0}^{R_{\pi,d}} \text{ for } \pi \in \hat\Gamma,\\[-0.9em]
& q_0 \leq 0, \, q_1 \leq 0,\nonumber\\
& w^s \,q_2(1- w^2/2) \leq 1 \text{ for } w \in [\sqrt{2-2U}, 2] \text{ if } 2 \nmid s,\nonumber\\
& (2-2u)^{s/2} q_2(u) \leq 1 \text{ for } u \in [-1, U] \text{ if } 2 \mid s,\nonumber\\[-0.4em]
& q_i|_{P_i} \leq 0 \text{ for } i = 3,\ldots,2t \Big\}.\nonumber
\end{align}

To describe $P_i$ as a semialgebraic set we first observe that by using the Gram decomposition of a positive semidefinite matrix, it can be written as
\[
P_i = \Big\{u \in \R^{\binom{i}{2}} : u_j \leq U \text{ for } j \in \left[\tbinom{i}{2}\right], \, \mathcal E(u) \succeq 0,\, \mathrm{rank}(\mathcal E(u)) \leq n \Big\},
\]
where $\mathcal E(u)$ is the symmetric $i \times i$-matrix with ones on the diagonal and the entries of $u$ in the upper and lower diagonal parts. Using Sylvester's criterion for positive semidefinite matrices we obtain the semialgebraic description
\begin{align}
\label{eq:semialgebraic descri}
P_i = \Big\{ u \in \R^{\binom{i}{2}} : \; & u_j \leq U \text{ for } j \in [\tbinom{i}{2}], \\[-0.15em]
& g(u) \geq 0 \text{ for } g \in G_{i,j} \text{ with } 2 \leq j \leq n,\nonumber\\[-0.3em]
& g(u) = 0 \text{ for } g \in G_{i,j} \text{ with } n+1 \leq j \leq i \Big\},\nonumber
\end{align}
where $G_{i,j}$ is the set of principal minors (the determinants of principal submatrices) of $\mathcal E(u) \in \R^{i \times i}$ of order $j$. This shows $E_{t,d}^*$ is a semidefinite program with semialgebraic constraints.

Here we make two observations which are important from a computational perspective in modeling the semialgebraic constraints as semidefinite constraints (see Section~\ref{sec:putinars theorem for invaraint}). The first observations is that $P_i$ is compact and that the constraints $g(u) \geq 0$ for $g \in G_{i,2}$  provide an ``algebraic certificate'' of this compactness. As explained in the following section, this means the above description is Archimedean, so that we can apply Putinar's theorem.

The second observation concerns additional symmetry in the semialgebraic constraints. The particles in an energy minimization problem are interchangeable, and this implies that the polynomials $p_1,\ldots,p_{2t}$ are not only invariant under the group $O(n)$, but also under (some) permutations of the $ni$ variables: We have
\[
p_i(x_1,\ldots,x_i) = p_i(x_{\sigma(1)}, \ldots, x_{\sigma(i)}) \quad \text{for all} \quad x_1,\ldots,x_i \in S^{n-1} \quad \text{and} \quad \sigma \in S_i.
\]
This implies we can choose the polynomials $q_1,\ldots,q_{2t}$ in such a way that they have additional symmetry. Let $\smash{\mathrm{Aut}^*(K_i)}$ be the edge-automorphism group of the complete graph $K_i$ on $i$ vertices, and let $\phi_i \colon S_i \to  \mathrm{Aut}^*(K_i)$ be the (not necessarily surjective) map that sends a permutation of the vertices of $K_i$ to the corresponding permutation of the edges. If $q_i$ is a polynomial that satisfies \eqref{eqpq}, then the polynomial
\begin{equation}
\label{eqsymmetrizeq}
\overline q_i(u_1,\ldots,u_r) = \frac{1}{i!}\sum_{\sigma \in \phi_i(S_i)} q_i(u_{\sigma(1)}, \ldots, u_{\sigma(r)}) 
\end{equation}
also satisfies \eqref{eqpq} and is invariant under the group $\phi_i(S_i)$. So, we may assume the polynomials $q_i$ to be $\phi_i(S_i)$-invariant. Since the sets $G_{i,j}$ used to define the semialgebraic sets $P_i$ are also invariant under this group, we say the semialgebraic constraints in the problem $E_{t,d}^*$ are $\phi_i(S_i)$-invariant.

\section{Sum of squares characterizations for invariant polynomials}
\label{sec:putinars theorem for invaraint}

In this section we first give some background on how Putinar's theorem can be used to approximate a semidefinite program with semialgebraic constraints by a sequence of block diagonal semidefinite programs. Then we show how symmetry in the polynomial constraints can be used to further block diagonalize these semidefinite programming formulations into smaller blocks. We show this can lead to significant computational savings by applying this to the problems $E_{2,d}^*$ from the previous section.

The \emph{quadratic module} generated by 
$
g_1,\ldots,g_m \in \R[x_1,\ldots,x_n]
$
is given by
\[
M(g_1,\ldots,g_m) = \Big\{ \sum_{i = 0}^m g_i s_i : s_0,\ldots,s_m \in \R[x_1,\ldots,x_n] \text{ SOS polynomials}\Big\},
\]
where $g_0$ denotes the constant one polynomial, and where a sum of squares (SOS) polynomial is a polynomial of the form $\sum_k p_k^2$, with $p_1,\ldots,p_K \in \R[x_1,\ldots,x_n]$.
Polynomials in $M(g_1,\ldots,g_m)$ are nonnegative on the \emph{basic closed semialgebraic} set
\[
S(g_1,\ldots,g_m) = \big\{x \in \R^n : g_i(x) \geq 0 \text{ for } i \in [m]\big\}.
\]
\emph{Putinar's theorem} \cite{Putinar1993a} says that under the condition that $M(g_1,\ldots,g_m)$ is Archimedean, every strictly positive polynomial on $S(g_1,\ldots,g_m)$ is contained in $M(g_1,\ldots,g_m)$. A quadratic module $M(g_1,\ldots,g_m)$ is  \emph{Archimedean} if it contains a polynomial $p$ such that $S(p)$ is compact. Such a polynomial $p$ is a simple algebraic certificate of the compactness of $S(g_1,\ldots,g_m)$.

For $\delta \in \N_0$, we define the \emph{truncated quadratic module} $M_\delta(g_1,\ldots,g_m)$ in the same way as we defined $M(g_1,\ldots,g_m)$, except now  we require each $s_i$ to have degree at most $2h_i$, where $h_i = \lfloor (\delta-\mathrm{deg}(g_i))/2 \rfloor$. Since higher degree terms can cancel each other out, the inclusion
\[
M_\delta(g_1,\ldots,g_m) \subseteq M(g_1,\ldots,g_m) \cap \R[x_1,\ldots,x_n]_\delta,
\]
can be strict. Here $\R[x_1,\ldots,x_n]_\delta$ denotes the vector space of polynomials of degree at most $\delta$. Putinar's theorem shows that each polynomial $p \in \R[x_1,\ldots,x_n]$ with $p(x_1,\ldots,x_n) > 0$ for all $(x_1,\ldots,x_n) \in S(g_1,\ldots,g_m)$ is contained in $M_\delta(g_1,\ldots,g_m)$ for all large enough $\delta$, and in \cite{NieSchweighofer2007} an upper bound on the smallest $\delta$ for which this is true is given in terms of the polynomials $g_1,\ldots,g_m$, the degree of $p$, and the minimum of $p$ over $S(g_1,\ldots,g_m)$.

Let $v_{h_i}(x)$ be a vector whose entries form a basis of $\R[x_1,\ldots,x_n]_{h_i}$. A polynomial of degree at most $2h_i$ is a sum of squares if and only if it can be written as
\[
s_i(x) = v_{h_i}(x)^{\sf T} Q_i  v_{h_i}(x),
\]
where $Q_i$ is a positive semidefinite matrix of size $\smash{\binom{n+h_i}{n}}$. (To prove $v_{h_i}(x)^{\sf T} Q_i v_{h_i}(x)$ is a sum of squares one can use a Cholesky factorization $Q_i = R_i^{\sf T} R_i$). This implies 
\[
M_\delta(g_1,\ldots,g_m) \simeq S_{\succeq 0}^{\binom{n+h_0}{n}} \times \cdots \times S_{\succeq 0}^{\binom{n+h_m}{n}}.
\] 

In a semidefinite program with semialgebraic constraints, we can now approximate a constraint of the form
\[
p(x_1, \ldots, x_n) \geq 0 \quad \text{for} \quad (x_1,\ldots,x_n) \in S(g_1,\ldots,g_m),
\]
by a degree~$\delta$ sum of squares characterization. By this we mean that we introduce additional positive semidefinite matrix variables $Q_0,\ldots,Q_m$ and replace the constraint $p|_{S(g_1,\ldots,g_m)} \geq 0$ by a set of linear constraints that enforces the identity
\[
p(x) = \sum_{i=0}^m g_i(x) v_{h_i}(x)^{\sf T} Q_i v_{h_i}(x).
\]
To obtain a set of linear constraints that enforces the above identity, we can express the left and right hand sides in terms of the same polynomial basis and equate the coefficients with respect to this basis. Here the basis choice for the entries of $v_i$ and the basis choice for the linear constraints can have great impact on the numerical conditioning of the resulting semidefinite program (see, for instance, \cite{LaatOliveiraVallentin2012}, where this happens for $2$-point bounds for sphere packing), but in the computations in this paper we only use the standard basis because we only use polynomials of low degree. 
%TODO fix ref

The approximations given by the semidefinite programs obtained in this way become arbitrarily good as we take sum-of-squares characterizations of higher degrees. Moreover, if a semidefinite program with semialgebraic constraints has an optimal solution where all the inequalities are strict, then the optimum is obtained for a finite degree sum-of-squares characterization.

We will show that if $p$ is invariant under the action of a group, then we can further block diagonalize the matrices $Q_i$. Let $\Gamma$ be a finite subgroup of $U(\C^n)$. This induces the unitary representation 
\begin{equation}
\label{eq:reprspol}
L \colon \Gamma \to U(\C[x_1,\ldots,x_n]), \, L(\gamma) p(x) = p(\gamma^{-1}x),
\end{equation}
where $\C[x_1,\ldots,x_n]$ has the inner product $\langle p, q\rangle = \sum_\alpha p_\alpha \overline{q_\alpha}$.
A polynomial $p$ is said to be $\Gamma$-invariant if $L(\gamma)p = p$ for all $\gamma \in \Gamma$, and a set of polynomials $\{g_1,\ldots,g_m\}$ is said to be $\Gamma$-invariant if 
\[
\{L(\gamma) g_1, \ldots, L(\gamma) g_m\} = \{g_1,\ldots,g_m\} \quad \text{for all} \quad \gamma \in \Gamma.
\]
Let $\Gamma_i$ be the stabilizer subgroup of $\Gamma$ with respect to $g_i$; that is, 
\[
\Gamma_i = \{\gamma \in\Gamma : L(\gamma)g_i = g_i\},
\]
where $\Gamma_0 = \Gamma$ because $g_0 = 1$. 

In the next proposition we show that if the polynomials $p$ and the set $\{g_1,\ldots,g_m\}$ are invariant under the group action, then the sum of squares polynomials can be taken to be invariant under the corresponding stabilizer subgroups. In \cite{RienerTheobaldAndrenLasserre2013} and  \cite{CimpriKuhlmannScheiderer2009} similar results for invariant polynomials are shown, although not in this form for an invariant set $\{g_1,\ldots,g_m\}$ of polynomials and using stabilizer subgroups.

\begin{proposition}
\label{prop:invariant putinar}
If $p \in M_\delta(g_1,\ldots,g_m)$ is $\Gamma$-invariant, then there are $\Gamma_i$-invariant sum of squares polynomials $s_i \in \R[x_1,\ldots,x_n]_{h_i}$ such that
\[
p = \sum_{i = 0}^m g_i s_i.
\]
\end{proposition}
\begin{proof}
Since $p$ is $\Gamma$-invariant we have
\[
p(x) = \frac{1}{|\Gamma|}\sum_{\gamma \in \Gamma} L(\gamma)p(x) = \frac{1}{|\Gamma|}\sum_{\gamma \in \Gamma} p(\gamma^{-1} x),
\]
and since $p$ lies in $M_\delta(g_1,\ldots,g_m)$, we have
\[
 \frac{1}{|\Gamma|}\sum_{\gamma \in \Gamma} p(\gamma^{-1} x) = \frac{1}{|\Gamma|}\sum_{\gamma \in \Gamma}\sum_{i=0}^m g_i(\gamma^{-1}x) s_i(\gamma^{-1}x).
 \]
Let $\Gamma_{i,j} = \{ \gamma \in \Gamma : L(\gamma) g_i = g_j \},$ so that, for each $0 \leq i \leq m$, we have $\Gamma_{i,i} = \Gamma_i$ and $\Gamma$ is the disjoint union of $\Gamma_{i,0},\ldots,\Gamma_{i,m}$.
Then,
\begin{align*}
 \frac{1}{|\Gamma|}\sum_{\gamma \in \Gamma}\sum_{i=0}^m g_i(\gamma^{-1}x) s_i(\gamma^{-1}x) 
&= \frac{1}{|\Gamma|} \sum_{j=0}^m \sum_{i=0}^m \sum_{\gamma \in \Gamma_{i,j}} g_i(\gamma^{-1} x) s_i(\gamma^{-1}x)\\
&= \frac{1}{|\Gamma|}\sum_{j=0}^m g_j(x) \sum_{i=0}^m\sum_{\gamma \in \Gamma_{i,j}} s_i(\gamma^{-1}x)
\end{align*}
So, if we define 
\[
\bar s_j(x) = \frac{1}{|\Gamma|} \sum_{i=0}^m\sum_{\gamma \in \Gamma_{i,j}} s_i(\gamma^{-1}x),
\]
then $p(x) = \sum_{j=0}^m g_j(x) \bar s_j(x)$. Since the cone of sum of squares polynomials is $\mathrm{GL}(\R[x])$-invariant, we see that the functions $\bar s_0, \ldots, \bar s_m$ are sums of squares polynomials. Moreover, for $\eta \in \Gamma_j$, we have
\[
L(\eta) \bar s_j(x) = \frac{1}{|\Gamma|} \sum_{i=0}^m\sum_{\gamma \in \Gamma_{i,j}} s_i(\gamma^{-1}\eta^{-1}x) = \frac{1}{|\Gamma|} \sum_{i=0}^m\sum_{\gamma \in \eta\Gamma_{i,j}} s_i(\gamma^{-1}x),
\]
so $\Gamma_i$-invariance of $\bar s_i$ follows from the identity 
\[
\Gamma_{j,k}\Gamma_{i,j} = \Gamma_{i,k} \quad \text{for all} \quad 0 \leq i,j,k \leq m.\qedhere
\]
\end{proof}

In \cite{GatermannParrilo2004} it is shown how the matrix used to represent an invariant sum of squares polynomial can be block diagonalized, and how this can be used to simplify semidefinite programs involving such polynomials. We combine this with Putinar's theorem and the above proposition to block diagonalize the representation of a positive invariant polynomial on an invariant semialgebraic set. To describe how this block diagonalization works we assume the group $\Gamma$ in \eqref{eq:reprspol}  consists of permutation matrices, because in this special case we can use the theory from Section~\ref{chap:energy:sec:explicit symmetry}, and in our application to $E_{t,d}^*$ all relevant groups are of this form. 

We can view the matrix $Q$ in a sum of squares representation
\[
s(x) = v_h(x)^{\sf T} Q v_h(x)
\]
as a positive definite kernel $[x_1,\ldots,x_n]_h \times [x_1,\ldots,x_n]_h \to \R$, where $[x_1,\ldots,x_n]_h$ is the set of monomials of degree at most $h$. The group $\Gamma$ has an obvious action on $[x_1,\ldots,x_n]_h$, and if $s$ is $\Gamma$-invariant, then we may assume $Q$ to be a $\Gamma$-invariant kernel: Represent $L(\gamma)$ in the monomial basis, so that $L(\gamma) v(x) = v(\gamma^{-1} x)$, and
\begin{align*}
s(x) 
&= \frac{1}{|\Gamma|}\sum_{\gamma \in \Gamma} s(\gamma^{-1} x) 
= \frac{1}{|\Gamma|}\sum_{\gamma \in \Gamma} v(\gamma^{-1} x)^{\sf T} Q v(\gamma^{-1} x)\\
&= \frac{1}{|\Gamma|}\sum_{\gamma \in \Gamma} v(x)^{\sf T} L(\gamma)^{\sf T} Q L(\gamma) v(x) 
= v(x)^{\sf T} \left(\frac{1}{|\Gamma|}\sum_{\gamma \in \Gamma} L(\gamma)^{\sf T} Q L(\gamma)\right) v(x), 
\end{align*}
which means we may replace $Q$ by its symmetrization 
$
1/|\Gamma|\sum_{\gamma \in \Gamma} L(\gamma)^{\sf T} Q L(\gamma).
$

By viewing $Q$ as a $\Gamma$-invariant kernel $[x_1,\ldots,x_n]_h \times [x_1,\ldots,x_n]_h \to \R$, we get
\[
Q(u, v) = \sum_{\pi \in \widehat\Gamma} \langle G_\pi, Z_\pi(u, v) \rangle, \quad \text{for} \quad u,v \in [x_1,\ldots,x_n]_h,
\]
where the $G_\pi$ are Hermitian positive semidefinite matrices. Here the $Z_\pi$ are the zonal matrices as defined in Section~\ref{subsecSymmetry adapted systems and zonal matrices}, where the topological space $X$ is now the finite set $[x_1,\ldots,x_n]_h.$  We have,
\begin{align*}
s(x) &= \sum_{u, v \in [x_1,\ldots,x_n]_h} Q(u, v) uv\\
&= \sum_{\pi \in \widehat\Gamma} \langle \widehat Q(\pi), \sum_{u, v \in [x_1,\ldots,x_n]_h } Z_\pi(u, v) uv \rangle = \sum_{\pi \in \widehat\Gamma} \langle \widehat Q(\pi), Z_\pi(x) \rangle,
\end{align*}
where we define the modified zonal matrices
\[
Z_\pi(x) = \sum_{u, v \in [x_1,\ldots,x_n]_h} Z_\pi(u,v) uv.
\]
In general we have to use Hermitian positive semidefinite blocks $G_\pi$, but in our computations all groups only have real irreducible representations, so, as explained at the end of Section~\ref{subsecSymmetry adapted systems and zonal matrices} we can use real blocks. Since the groups here are finite, we can compute the symmetry adapted system algorithmically. Before we explain how this is done we note that we can compute the size $m_\pi$ of $G_\pi$ without having to compute the actual block diagonalization. 

For this we first notice that $m_\pi$ now denotes the number of representations in an orthogonal decomposition of $\R[x]_h$ into irreducible unitary representations that are unitarily equivalent to $\pi$. First observe that 
$
m_\pi = m_\pi(0) + \cdots + m_\pi(h),
$
where $m_\pi(k)$ denotes the denotes the number of representations in an orthogonal decomposition of $\R[x]_{=k}$ into irreducible unitary representations that are unitarily equivalent to $\pi$.  Here $\R[x]_{=k}$ is the space of homogeneous polynomials of degree $k$. We can compute the numbers $m_\pi(k)$ by a theorem of Molien \cite{Molien1897}, which gives the equality of the formal power series 
\[
\sum_{k=0}^\infty m_\pi(k) t^k = \frac{1}{|\Gamma|} \sum_{\gamma \in \Gamma} \frac{\mathrm{trace}(\pi(\gamma))}{\mathrm{det}(I - t L(\gamma))}.
\]

To compute the actual block diagonalization we use the projection algorithm as described in \cite{Serre1977}, which generates the symmetry adapted systems used to construct the zonal matrices. This algorithm works as follows: First define the operators
\[
p_{j,j'}^\pi = \frac{d_\pi}{|\Gamma|} \sum_{\gamma\in\Gamma} \pi(\gamma^{-1})_{j,j'} L(\gamma).
\]
Then, for each $\pi \in \hat{\Gamma}$, we choose a basis $\{e_{\pi,i,1}\}_{i=1}^{m_\pi}$ of $\mathrm{Im}(p_{1,1}^\pi)$, and for every $\pi$, $i$, $j$ we set
$
e_{\pi,i,j} = p_{j,1}^\pi e_{\pi,i,1}.
$
In \cite{Serre1977} it is shown that this yields a (nonorthonormal) symmetry adapted system, and if we choose the bases $\{e_{\pi,i,1}\}_{i=1}^{m_\pi}$ of $\smash{\mathrm{Im}(p_{j,j'}^\pi)}$ to be orthonormal, then it is not difficult to show the resulting symmetry adapted system is also orthonormal.

We apply the above techniques to the problems $E_{2,d}^*$ from Section~\ref{sec:reduction to sdps} for Riesz $s$-energy problems. First we show how to model the semialgebraic constraints using sum of squares characterizations without exploiting the additional symmetry. For the constraints $q_3|_{P_3} \leq 0$ and $q_4|_{P_4} \leq 0$ we use sum of squares characterizations of degree $\delta$, and we denote the resulting semidefinite program by $\smash{E_{2,d,\delta}^*}$. 

If $s$ is odd, we model the semialgebraic constraint 
\[
w^s \,q_2(1- w^2/2) \leq 1 \quad \text{for} \quad w \in [\sqrt{2-2U}, 2]
\]
by introducing positive semidefinite matrices $Q_{2,1}$ and $Q_{2,2}$ and adding a set of linear constraints to enforce the identity
\[
w^s \,q_2(1- w^2/2) =  1 + (w - \sqrt{2-2U}) \, v_{h_1}(w)^{\sf T} Q_{2,1} v_{h_1}(w) +(2 - w)  v_{h_1}(w)^{\sf T} Q_{2,2} v_{h_1}(w).
\]
As explained above, we know that for sufficiently large $h_1$, this sum-of-squares constraint approximates the above semialgebraic constraint arbitrarily well. However, for this special case, where we have an odd degree polynomial that is nonnegative on a compact interval, a result of Luk\'acs (see, for instance, \cites{PowersReznick2000,PolyaSzego1998}) says that for $h_1 = (2d+s-1)/2$ the above semialgebraic constraint and sum-of-squares constraint are identical.

If $s$ is even, we model the semialgebraic constraint 
\[
(2-2u)^{s/2} q_2(u) \leq 1 \quad \text{for} \quad u \in [-1, U]
\]
by introducing positive semidefinite matrices $Q_{2,1}$ and $Q_{2,2}$ and adding a set of linear constraints to enforce the identity
\[
(2-2u)^{s/2} q_2(u) = 1 + v_{h_2}(u)^{\sf T} Q_{2,1} v_{h_2}(u) + (u-1)(U - u)  v_{h_3}(u)^{\sf T} Q_{2,2} v_{h_3}(u),
\]
if $s/2 + d$ is even, and
\[
(2-2u)^{s/2} q_2(u) = 1+ (u -1) v_{h_4}(u)^{\sf T} Q_{2,1} v_{h_4}(u) +  (U - u) v_{h_4}(u)^{\sf T} Q_{2,2} v_{h_4}(u),
\]
if $s/2 + d$ is odd. By the same result of Luk\'acs mentioned above, the above semialgebraic constraint and sum of squares constraint are identical if we take $h_2 = (s/2+d)/2$, $h_3 = (s/2+d)/2 - 1$, and $h_4 = (s/2+d-1)/2$.

We model the semialgebraic constraint $q_3|_{P_3} \leq 0$ by introducing positive semidefinite matrix variables $Q_{3}$, $Q_{3,2,g}$ for $g \in G_{3,2}$, and $Q_{3,3,g}$ for $g \in G_{3,3}$, and adding a set of linear constraints to enforce the identity
\begin{align}
\label{eq:putinar rep 1}
&q_3(u) + v_{h'_1}(u)^{\sf T} Q_3 v_{h'_1}(u) + \sum_{k \in \{2,3\}} \, \sum_{g \in G_{3,k}} g(u) v_{h'_k}(u)^{\sf T} Q_{3,g} v_{h'_k}(u) = 0,
\end{align}
where $h_1' = \lfloor \delta/2 \rfloor$, $h_2' = \lfloor (\delta-2)/2 \rfloor$, $h_3' = \lfloor (\delta - 3)/2 \rfloor$.

In the semialgebraic description of $P_4$ in \eqref{eq:semialgebraic descri} we do not just have polynomial inequalities constraints but also the polynomial equality constraint $\det(\mathcal E(u)) = 0$. We could replace this equality constraint by two inequality constraints, but this would be computationally inefficient. Instead, we introduce positive semidefinite matrix variables $Q_4$, $Q_{4,2,g}$ for $g \in G_{4,2}$, $Q_{4,3,g}$ for $g \in G_{4,3}$, and $q_{\alpha,\pm 1} \in \R_{\geq 0}$ for $\alpha \in \N_0$ with $\|\alpha\|_1 = \sum_i \alpha_i \leq \delta-6$, and use the sum of squares characterization
\begin{align}
\label{eq:putinar representation}
& q_4(u) + v_{h_1'}(u)^{\sf T} Q_4 v_{h_1'}(u) + \sum_{k \in \{2,3\}} \, \sum_{g \in G_{4,k}} g(u) v_{h_k'}(u)^{\sf T} Q_{4,g} v_{h_k'}(u)\\
&\qquad + \det(\mathcal E(u)) \sum_{\alpha \in \N_0^6 : \|\alpha\|_1 \leq \delta-6} (q_{\alpha,1} - q_{\alpha,-1}) u_1^{\alpha_1} \cdots u_6^{\alpha_6} = 0.\nonumber
\end{align}

\begin{table}[!h]
\centering
\scriptsize
\begin{tabular}{lllllll}
\toprule
 &\multicolumn{3}{c}{Without symmetry reduction} & \multicolumn{3}{c}{With symmetry reduction}\\
\cmidrule(lr){2-4} \cmidrule(lr){5-7}
$\delta$ & $Q_4$ & $Q_{4,2,g}$ & $Q_{4,3,g}$ & $Q_4$ & $Q_{4,2,g}$ & $Q_{4,3,g}$\\
\midrule
0 & 1 & 0 & 0 & 1 & 0 & 0\\
1 & 1 & 0 & 0 & 1 & 0 & 0\\
2 & 7 & 1 & 0 & 2 & 1 & 0\\
3 & 7 & 1 & 1 &  2& 1 & 1\\
4 & 28 & 7 & 1 & 5 &4 &1\\
5 & 28 & 7 & 7 & 5 &4 & 3\\
6 & 84 & 28 & 7 & 12 & 13 & 3\\
7 & 84 & 28 & 28 & 12 & 13 & 9\\
8 & 210 & 84 & 28 & 29 & 33 & 9\\
9 & 210 & 84 & 84 & 29 & 33 & 27\\
10 & 462 & 210 & 84 & 63 & 75 & 27\\
11 & 462 & 210 & 210 & 63 & 75 & 69\\
12 & 924 & 462 & 210 & 124 & 153 & 69\\
13 & 924 & 462 & 462 & 124 & 153 & 153\\
14 & 1716 & 924 & 462 & 228 & 291 & 153\\
15 & 1716 & 924 & 924 & 228 & 291 & 306\\
16 & 3003 & 1716 & 924 & 395 & 519 & 306\\
17 & 3003 & 1716 & 1716 & 395 & 519 & 570\\
18 & 5005 & 3003 & 1716 & 654 & 882 & 570\\
19 & 5005 & 3003 & 3003 & 654 & 882 & 999\\
20 & 8008 & 5005 & 3003 & 1040 & 1435 & 999\\
\bottomrule
\end{tabular}
\vspace{2ex}
\caption{Block sizes in the sum-of-squares modeling of the $i=4$ constraints in $E_{2,d,\delta}^*$ with and without symmetry reduction.}
\label{table:blocksizes}
\vspace{-1em}
\end{table}

Now we show by how much we can reduce the largest block size in the semidefinite program $E_{2,d,\delta}^*$ by exploiting the symmetry in the semialgebraic constraints. The matrix 
\[
Q_4 \in S_{\succeq 0}^{\binom{6+\lfloor \delta/2 \rfloor}{6}} 
\]
from \eqref{eq:putinar representation} typically forms the largest block in $E_{2,d,\delta}^*$. This block is larger than any other matrix used in the sum-of-squares modeling, and unless $d$ is much larger than $\delta$, it is larger than any of the $F_\pi$ blocks. As explained at the end of Section~\ref{sec:reduction to sdps}, the polynomial $q_4$ and the set of polynomials in the semialgebraic description of $P_4$ are invariant under the group $\Gamma = \phi_4(S_4)$. The stabilizer subgroup $\Gamma_1$ of $\Gamma$ with respect to the constant $1$ polynomial is isomorphic to $S_4$. The stabilizer subgroup $\Gamma_g$ of $\Gamma$ with respect to a polynomial $g \in G_{4,k}$ is isomorphic to the Klein-Four group for $k = 2$ and to $S_3$ for $k = 3$. In Table~\ref{table:blocksizes} we first show the size of $Q_4$, $Q_{4,2,g}$, and $Q_{4,3,g}$ for different values of $\delta$ for the case where we are not using symmetry. We use a Magma \cite{BosmaCannonPlayoust1997} implementation of the Molien series mentioned above to compute the blocksizes that we get when we do exploit the symmetry. In Table~\ref{table:blocksizes} we then show the largest of these block sizes when block diagonalizing the matrices $Q_4$, $Q_{4,2,g}$, and $Q_{4,3,g}$. In our computation we will use $\delta = 6$ and $\delta = 8$, where we see this gives a $6$ fold reduction in the largest block size in the semidefinite program.

\section{Computations and discussion}
\label{sec:enegy:computations}

Our goal here is to show how the $4$-point bound $E_2^*$ can be computed for Riesz $s$-energy problems on $S^2$, and to observe that these bounds are numerically (with high precision) sharp for $N = 5$ and $s = 1,2,\ldots,7$. To do this we develop a program that can generate the semidefinite programs $E_{2,d,\delta}^*$ from Section~\ref{sec:putinars theorem for invaraint}. We solve these semidefinite programs for $d=6$ and $\delta = 6,8$ with a high precision semidefinite programming solver. Then we check that the optimal objective values given by the solver (consisting of $28$ decimal digits) coincide with the first $28$ decimals of the Riesz $s$-energy 
\begin{equation}
\label{energy:tirangular bipiramid}
\frac{6}{2^{s/2}} + \frac{3}{3^{s/2}} + \frac{1}{4^{s/2}},
\end{equation}
of the triangular bipiramid.

We implement the program in  Julia  \cite{Bezanson2014}, which is a high level language that has a modern type system and JIT compiler for fast execution of the code. We first generate the symmetry adapted system and the zonal matrices as described in Section~\ref{subsect:explicit harmonic analysis}. For this we develop a simple Julia library for sparse multivariate polynomials, which includes generators for the Laplace spherical harmonics in cartesian coordinates and a generator for the Clebsch--Gordan coefficients. To generate high precision solver input we perform all computations in high precision arithmetic using the MPFR library \cite{Fousse2007}.

To compute the polynomials $q_0,\ldots,q_{2t}$ from Section~\ref{sec:reduction to sdps}, we write the polynomials from \eqref{eq:a2polynomial} as polynomials in the inner products. For this we need to solve a large number of instances of the following problem: Suppose $p \in \R[x_1,\ldots,x_i]_{2d}$, where each $x_k$ is a vector of $3$ variables, is $O(3)$-invariant. We want to find a polynomial $q \in \R[u_1,\ldots,u_r]$, with $r = \binom{i+1}{2}$, such that
\[
p(x_1,\ldots,x_i) = q(x_1 \cdot x_1, x_1 \cdot x_2, \ldots, x_{i} \cdot x_i).
\]
As mentioned in Section~\ref{sec:reduction to sdps}, by a theorem from invariant theory such a polynomial $q$ is guaranteed to exist. If $m \in \R[u_1,\ldots,u_r]$ is a monomial, then the polynomial 
\[
m(x_1 \cdot x_1, x_1 \cdot x_2, \ldots, x_i \cdot x_i)
\]
is homogeneous of degree $2\deg(m)$. This means we may assume $\deg(q) \leq d$. We construct a linear system $A x = b$, where the rows of $A$ and $b$ are indexed by the monomials in $3i$ variables of degree at most $2d$, and the columns of $A$ and rows of $x$ by the monomials in $s$ variables up to degree $d$. The size of $A$ increases rapidly: for $i=4$ and $d=6$ it has about $2.7$ million rows. The matrix is sparse, however, where the maximum number of nonzeros in a row is $3^d$, and although this is exponential in $d$, for $d=6$ this is just $729$. We therefore store $A$ in a sparse data structure. For $i=4$, the system $Ax = b$ has more rows than columns. So we use a least squares approach and solve $A^{\sf T} A x = A^{\sf T} b$ instead. The matrix $A$ is in general not of full column rank (in general, $q$ is not unique), which means $A^{\sf T} A$ is singular, so instead we solve the system $(A^{\sf T} A + \varepsilon I) x = A^{\sf T} b$, where $\varepsilon> 0$ is small. Because a high precision solver that can work with sparse data structures is not readily available, we implement a simple pivoting, sparse, high precision, Cholesky factorization algorithm. We use this to compute the Cholesky factorization $A^{\sf T} A + \varepsilon I = P R^{\sf T} R P^{\sf T}$, where $P$ is a permutation matrix, and retrieve $x$ using backwards substitution. Finally, we use the equation relating $p$ and $q$ to verify the correctness of the computed polynomial up to a large number of digits. We then use equation \eqref{eqsymmetrizeq} to symmetrize the polynomial $q$.

We implement a GAP~\cite{BosmaCannonPlayoust1997} script to generate the symmetry adapted systems used in Section~\ref{sec:putinars theorem for invaraint} for the symmetrized sum-of-squares characterizations. For this we need the orthogonal (real unitary) irreducible representations of the relevant stabilizer subgroups of the symmetric groups $S_3$ and $S_4$. Here, the only groups with nonobvious irreducible representations are the symmetric groups themselves, and we use Young's orthogonal form (see, for instance, \cite{Beveren2002}) for these representations.

We develop a semidefinite programming specification library in Julia that allows for modeling polynomial equality constraints involving multivariate polynomials. We use this together with the above mentioned code to generate the semidefinite programs and output these in the SDPA-sparse format \cite{Mittelman}. 

Just as we did for the variables $q_{\alpha,\pm 1}$ from equation \eqref{eq:putinar representation}, we model the free variables $a_0,\ldots,a_5$ from \eqref{eq:etd} as the difference of two $1 \times 1$ positive semidefinite matrices. This, however, means that the resulting semidefinite programs are unbounded, which implies the dual programs are not strictly feasible. That is, the dual programs do not admit feasible solutions where all blocks are positive definite. For many semidefinite programming solvers this is a problem, and, in particular, the high precision solvers SDPA-QD and SDPA-GMP cannot be used in this situation. We therefore add a new parameter $M$, and we constrain the variables used to model the free variables to be at most $M$.  In this way the dual problems become strictly feasible and can be solved with high precision solvers. Notice that for any value of $M$ it is guaranteed we get a lower bound on the energy, and if $M$ is large enough (we use $M = 1000$) this does not weaken the bound. 

We model the polynomial equality constraints \eqref{eq:putinar rep 1} and \eqref{eq:putinar representation} by a linear constraint for each monomial. When we use the additional symmetry from Section~\ref{sec:putinars theorem for invaraint}, then this results in linearly dependent constraints. For some solvers such as the machine precision solver CSDP \cite{Borchers1999} this is not a problem, but solvers from the SDPA family do not work well in this case. Therefore we first remove identical constraints and then use a QR factorization of the constraint matrix of the semidefinite program to remove any remaining linearly dependent constraints.

The solver SDPA-QD works with quad double precision, which means solving a semidefinite program with this solver yields a solution with approximately $28$ decimal digits of precision. In all sharp instances that we compute we verify that at least the first $28$ decimal digits given by the solver agree with the first $28$ decimals of the energy of a configuration. It is important to observe here that to get more digits we do not have have to increase the parameters $d$ and $\delta$ -- we can use the exact same semidefinite program -- but we simply have to increase the precision parameter in the Julia code that generates the solver input and increase the precision parameter in the solver (where we switch to SDPA-GMP for variable precision instead of SDPA-QD). For our purposes, however, there is no reason to use higher precision. Notice that this very different from the sphere packing problem, which was recently solved in dimensions $8$ and $24$ using $2$-point bounds \cites{Viazovska2016a,Cohn2016b}, where one needs to increase the number of terms in the inverse Fourier transform (which corresponds to increasing $d$ in  $E_{t,d,\delta}^*$) for the bound to get closer to the exact optimal value.

As is to be expected, the computation time increases strongly with $\delta$. Computing the bound $E_{2,6,6}^*$ with SDPA-QD takes approximately $10$ minutes (on a standard desktop computer) if we do use the additional symmetry from Section~\ref{sec:putinars theorem for invaraint}, and takes approximately $80$ minutes if we do not use this additional symmetry. Computing the bound $E_{2,6,8}^*$ takes approximately $7$ hours with additional symmetry and $150$ hours without additional symmetry. Here, the value of $s$ itself has virtually no impact on the computation time, however, as $s$ increases we do need to increase the parameter $\delta$ to get a sharp bound. We observe that for $s = 1,\ldots,5$, the bound $E_{2,6,6}^*$ is numerically sharp, and for $s=6,7$ the bound $E_{2,6,8}^*$ is numerically sharp; that is, the $28$ decimal digits given by the solver agree with \eqref{energy:tirangular bipiramid}.

To find new sharp instances, it would also be of interest to compute $E_2$ for energy minimization problems on higher dimensional spheres, or other compact spaces. As observed in \cite{CohnWoo2011}, of particular interest is the case of $24$ particles on $S^3$, as here the $24$-cell seems to be optimal for some potential functions, but for other potentials the optimal configurations seem to be more exotic. It would be remarkable if $E_2$ would be universally sharp for $24$ particles on $S^3$. It would also be interesting to use the techniques developed in this paper to compute $4$-point bounds for packing problems such as spherical code problems on $S^{n-1}$. Of particular interest would be the spherical code problem $A(4,\arccos(1/3))$, where a construction of $14$ points exists, and where the $2$ and $3$-point bounds give the upper bounds $16$ and $15$ \cite{BachocVallentin2008}.

\appendix

\section{Invariant positive definite kernels} 
\label{sec:invariant posdef kernels}

In this appendix we prove theorems concerning the ``simulateneous block diagonalization'' of invariant positive definite kernels. These results are used in Section~\ref{subsecSymmetry adapted systems and zonal matrices}, where we also give more background information and introduce some of the notation used in this appendix. 

The first theorem characterizes the extreme rays of the cone of invariant positive definite kernels. As a special case, this shows 
\[
\partial_r \big(\Ccal(X \times X; \C)_{\succeq 0}\big) = \Big\{ f \otimes \bar f : f \in \Ccal(X; \C)\Big\},
\]
where we use the notation $\partial_r$ for the extreme rays of a cone. This theorem and its proof are a generalization to kernels of a result in harmonic analysis about functions of positive type as given in \cite{Folland1995}. In this appendix $X$ is a compact metric space with a continuous action of a compact group $\Gamma$ (in the following theorem, however, we may assume $X$ and $\Gamma$ to be locally compact).

\begin{theorem}
\label{kernels:theorem:extreme rays}
We have
\[
\partial_r\big(\Ccal(X \times X; \C)_{\succeq 0}^\Gamma\big) = \Big\{K_\varphi : \pi \in \hat\Gamma, \, \varphi \in \mathrm{Hom}_\Gamma(X,\Hcal_\pi)\Big\},
\]
where $\Hcal_\pi$ is the Hilbert space of the representation $\pi \in \hat\Gamma$, where $\mathrm{Hom}_\Gamma(X,\Hcal_\pi)$ is the space of $\Gamma$-equivariant maps $X \to \Hcal_\pi$, and where $K_\varphi$ is defined by 
\[
K_\varphi(x,y) = \big\langle \varphi(x), \varphi(y) \big\rangle \quad \text{for all} \quad  x,y \in X.
\]
\end{theorem} 
\begin{proof}
Let $K$ be a nonzero kernel in $\smash{\Ccal(X \times X; \C)_{\succeq 0}^\Gamma}$. As shown in \cite{CohnWoo2011}, we can use a Gelfand-Naimark-Segal type construction to build a unitary representation $\pi \colon \Gamma \to U(\Hcal_\pi)$ and a nonzero function $\varphi \in \mathrm{Hom}_\Gamma(X,\Hcal_\pi)$, so that
\[
K(x, y) = \big\langle \varphi(x), \varphi(y) \big\rangle \quad \text{for all} \quad x,y \in X.
\]
Indeed, let $\C^X$ be the complex vector space of formal linear combinations of elements in $X$, and define the subspace $N = \mathrm{span} \{ x \in X : K(x,x) = 0\}$. Define an inner product on the quotient space $\C^X/N$ by setting $\langle x+N, y+N \rangle = K(x,y)$ for all $x,y \in X$ and extending by (anti)linearity. The action of $\Gamma$ on $X$ extends to the homomorphism $\pi \colon \Gamma \to U(\Hcal_\pi)$, where $\Hcal_\pi$ is the Hilbert space obtained by completing $\C^X/N$ in the metric defined by the inner product $\langle \cdot, \cdot \rangle$. Here $\pi(\gamma)$ is an isometry because $K$ is $\Gamma$-invariant, and because $\pi(\gamma)$ is invertible, it is a unitary operator.
Since $\langle \pi(\gamma) x + N, y + N \rangle = K(\gamma x, y)$, it follows from $K$ being continuous and the action of $\Gamma$ on $X$ being continuous, that the map $\gamma \to \langle \pi(\gamma) x + N, y + N\rangle$ is continuous. So $\pi$ is a unitary representation. We define the $\Gamma$-equivariant map $\varphi \colon X \to \Hcal$ by $\varphi(x) = x+N$. This map is continuous, because
\[
\|\varphi(y) - \varphi(x)\|^2 
\leq K(x,x)+K(y,y)-K(x,y)-K(y,x),
\]
and, moreover, $\varphi$ is injective and has dense span.

Now assume $K$ spans an extreme ray. If $\pi$ is reducible, then $\Hcal_\pi$ admits a nontrivial orthogonal decomposition $\Mcal_1 \oplus \Mcal_2$ into $\Gamma$-invariant subspaces. Let $P_i \colon \Hcal \to \Mcal_i$ be the projector onto $\Mcal_i$, and set $\varphi_i = P_i \circ \varphi$. Let 
\[
K_i(x,y) = \langle \varphi_i(x), \varphi_i(y) \rangle \quad \text{for} \quad x,y \in X,
\]
so that $K = K_1 + K_2$. Now we show the kernels $K_1$ and $K_2$ do not lie on the same ray: If $K_2 = |c|^2 K_1$ for some nonzero $c \in \C$, then we can define a $\Gamma$-equivariant unitary operator $T \colon \Mcal_2 \to \Mcal_1$ by setting $T(\varphi_2(x)) = c \, \varphi_1(x)$ for all $x \in X$. But this implies $\varphi = \varphi_1 + \varphi_2 = c^{-1} T \circ \varphi_2 + \varphi_2$, and this contradicts with $\varphi$ being injective and having dense span in $\Hcal_\pi$. Therefore, $\pi$ must be irreducible.

Now assume $K$ is a nonzero kernel in $\smash{\Ccal(X \times X; \C)_{\succeq 0}^\Gamma}$, such that
\[
K(x, y) = \big\langle \varphi(x), \varphi(y) \big\rangle, \quad \text{for all} \quad x,y \in X,
\]
for some irreducible unitary representation $\pi \colon \Gamma \to U(\Hcal_\pi)$ and $\varphi \in \mathrm{Hom}_\Gamma(X,\Hcal_\pi)$.
Let $K_1$ and $K_2$ be two kernels in $\Ccal(X \times X; \C)_{\succeq 0}^\Gamma$ with $K = K_1 + K_2$.
We have $K_1(x,x) = K(x,x)-K_2(x,x)\leq K(x,x)$ for all $x \in X$, so 
\[
|K_1(x,y)| \leq K_1(x,x)^{1/2}K_1(y,y)^{1/2} \leq K(x,x)^{1/2}K(y,y)^{1/2} \quad \text{for all} \quad x,y\in X.
\]
This means we can define the bounded Hermitian form $\langle \cdot, \cdot \rangle_1$ on $\C^X / N$ by setting 
\[
\langle \varphi(x) + N, \varphi(y) + N \rangle_1 = K_1(x,y)
\]
and extending by (anti)linearity. The form $\langle \cdot, \cdot \rangle_1$ is continuous since it is bounded, so we can extend it to the Hilbert space $\Hcal_\pi$. By the Riesz representation theorem for Hilbert spaces there is a bounded self-adjoint operator $T$ on $\Hcal_\pi$ such that 
\[
\langle \varphi(x) + N, \varphi(y) + N \rangle_1 = \langle T (\varphi(x) + N) , \varphi(y)+N \rangle \quad \text{for all} \quad x,y \in X.
\]
This operator is $\Gamma$-equivariant: For all $x,y \in X$ and $\gamma \in \Gamma$ we have
\begin{align*}
\langle T \pi(\gamma) (\varphi(x)+N), \varphi(y)+N \rangle
&= \langle \varphi(\gamma^{-1}x)+N, \varphi(y)+N\rangle_1\\
&=  K_1(\gamma^{-1}x, y)
= K_1(x, \gamma y) \\
&= \langle \varphi(x), \varphi(\gamma y)\rangle_1
= \langle T \varphi(x), \pi(\gamma^{-1})\varphi(y)\rangle\\
&= \langle \pi(\gamma) T \varphi(x), \varphi(y)\rangle.
\end{align*}
Since $\pi$ is irreducible, Schur's lemma states there is a $c \in \C$ such that $T = cI$. But this means that 
\[
K_1(x,y)= \langle \varphi(x) +N, \varphi(y)+N \rangle_1 = \langle T(\varphi(x)+N), \varphi(y)+N \rangle = cK(x,y),
\]
for all $x,y \in X$, and hence $K_1 = cK$ and $K_2 = (1-c)K$, which shows that $K$ spans an extreme ray. 
\end{proof}

Next, we prove the existence of a symmetry adapted system. For this we first need a few lemmas.

\begin{lemma}
\label{lem:existence of strictly positive measure}
The space $X$ admits a strictly positive, $\Gamma$-invariant, Radon probability measure.
\end{lemma}
\begin{proof}
Let $\{x_i\}$ be a dense sequence in $X$ and let $\{a_i\}$ be a sequence of strictly positive numbers that sum to one. Define a Borel probability measure $\mu$ by setting
$
\mu(U) = \sum_{i : x_i \in U} a_i
$ 
for all Borel sets $U$.  The desired measure is obtained by averaging $\mu$ over the Haar measure of $\Gamma$.
\end{proof}

We say that a sequence $\{I_n\}$ of kernels in $\Ccal(X \times X; \C)$ is an \emph{approximate identity} of $X$ if $\|T_{I_n} f -f\|_\infty \to 0$ as $n \to \infty$ for each $f \in \Ccal(X; \C)$, where
\[
T_K \colon \Ccal(X; \C) \to \Ccal(X; \C), \, T_Kf(x) = \int K(x,y) f(y) \, d\mu(y),
\]
and where $\mu$ is some fixed strictly positive $\Gamma$-invariant Radon probability measure.

\begin{lemma}\label{lem:approximate identity}
The space $X$ admits an approximate identity $\{I_n\}$, where each kernel $I_n$ may be assumed to be real-valued, symmetric, and $\Gamma$-invariant.
\end{lemma}
\begin{proof}
Let $d$ be a compatible metric on $X$. Let $\{U_i^1\}, \{U_i^2\}, \ldots$ be a sequence of finite open covers of $X$ such that for all $i$ and $n$ the diameter of $U_i^n$ is at most $1/n$. For each $i$ and $n$ inductively select a compact set $C_i^n \subseteq U_i^n$ such that \[\mu(U_i^n \setminus C_i^n) \leq \mu(C_i^n)/n,\] (this is possible by inner regularity of $\mu$), and remove $C_i^n$ from the sets $U_j^n$ for $j \neq i$. We then have $C_i^n \cap U_{i'}^n = \emptyset$ for all $n$ and all distinct $i$ and $i'$. 

Let $\{p_i^n\}_i$ be a partition of unity subordinate to the cover $\{U_i^n\}_i$, so that the restriction of $p_i^n$ to $C_i^n$ is identically $1$, and define the kernel $K_n \in \Ccal(X \times X)$ by the finite sum
\[
K_n(x, y) = \sum_i \frac{p_i^n(x)p_i^n(y)}{\mu(C_i^n)}.
\]

Let $f \in \Ccal(X; \C)$ and $\varepsilon > 0$. For large enough $n$ we have
\[
\mu(U_i^n \setminus C_i^n) \leq \frac{\mu(C_i^n) }{2\|f\|_\infty}\varepsilon \quad \text{and} \quad \sup_{x,y\in C_i^n} |f(x)-f(y)| \leq \frac{1}{2} \varepsilon\quad \text{for all} \quad i.
\]
Then for each $x \in X$,
\[
|T_{K_n} f(x) - f(x)| = \left| \sum_i \int_{U_i^n} \frac{p_i^n(x)p_i^n(y)}{\mu(C_i)}f(y) \,d\mu(y) - f(x) \right| \leq A + B,
\]
where
\begin{align*}
A &= \left|\sum_i \int_{C_i^n} \frac{p_i^n(x)p_i^n(y)}{\mu(C_i^n)}f(y) \,d\mu(y) - f(x) \right|\\
&= \left|\sum_i \frac{p_i^n(x)}{\mu(C_i^n)} \int_{C_i^n} |f(y) - f(x) |\,d\mu(y)\right| \leq \sum_i p_i^n(x) \frac{\varepsilon}{2} = \frac{\varepsilon}{2}
\end{align*}and
\begin{align*}
B &= \left|\sum_i \int_{U_i^n\setminus C_i^n} \frac{p_i^n(x)p_i^n(y)}{\mu(C_i^n)}f(y) \,d\mu(y)\right|\\
&= \sum_i \frac{p_i^n(x)}{\mu(C_i^n)} \int_{U_i^n\setminus C_i^n} p_i^n(y) |f(y)|\, d\mu(y)
= \sum_i p_i^n(x) \frac{\mu(U_i^n \setminus C_i^n)}{\mu(C_i^n)} \|f\|_\infty \leq \frac{\varepsilon}{2}.
\end{align*}
So, for each $\varepsilon > 0$ we have $\|T_{K_n} f- f\|_\infty \leq \varepsilon$ for sufficiently large $n$, which means that the sequence $\{K_n\}$ is an approximate identity. 
	
Let 
\[
I_n(x,y) = \int_\Gamma K_n(\gamma x, \gamma y) \,d\gamma,
\]
where we integrate against the normalized Haar measure of $\Gamma$. Then $\{I_n\}$ is an approximate identity, and each $I_n$ is real-valued, symmetric, and $\Gamma$-invariant.
\end{proof}

We need the following part of the Peter--Weyl theorem.  A proof for the case where a compact group acts on itself can be found in for instance \cite{Folland1995}, and a generalization of this to the setting of a compact group acting on a compact metric space can be found in \cite{Laat2016}.

\begin{lemma}\label{lem:peter-weyl core}
The space $\Ccal(X; \C)$ is equal to the closure of the sum of its finite dimensional $\Gamma$-invariant subspaces.
\end{lemma}

We also need the following variation on the Schur orthogonality relations, for which a proof can be found in \cite{Vallentin2008}.

\begin{lemma}
\label{lem:invariant inner product}
Let $\pi \colon \Gamma \to U(\Hcal)$ be a unitary representation, and let $\langle \cdot, \cdot \rangle$ be a $\Gamma$-invariant sesquilinear from on $\Hcal$. Let $\Mcal$ and $\Mcal'$ be finite-dimensional, irreducible subrepresentations with orthonormal bases $\{e_i\}$ and $\{e_j'\}$.
\begin{enumerate}
\item If $\Mcal$ and $\Mcal'$ are not equivalent, then $\langle e_i, e_j'\rangle = 0$ for all $i$ and $j$.
\item If there exists a $\Gamma$-equivariant bijection $T \colon \Mcal \to \Mcal'$ such that $Te_i=e_i'$ for all $i$, then there is a $c \in \mathbb{C}$ such that $\langle e_i, e_j' \rangle = c \delta_{i,j}$ for all $i$ and $j$.
\end{enumerate}  
\end{lemma}

\begin{theorem}
\label{thm:symmetry adapted system}
Let $X$ be a compact metric space with a continuous action of a compact group $\Gamma$. The space $X$ admits a symmetry adapted system.
\end{theorem}
\begin{proof}
Let $C$ be the set of all linearly independent sets of nontrivial, finite dimensional, $\Gamma$-invariant subspaces of $\Ccal(X; \C)$. Here, a set $S$ of subspaces is said to be \emph{linearly independent} if for any $n \in \N$ and distinct $A,B_1,\ldots,B_n \in S$, the intersection of $A$ with the sum $B_1 + \cdots + B_n$ is the zero space. This is equivalent to requiring that the union of any set of bases of the  subspaces in $S$ is linearly independent. If $X$ is nonempty, then $\Ccal(X; \C)$ is nonempty, so by Lemma~\ref{lem:peter-weyl core} the set $C$ is nonempty. 
	
Define a partial order on $C$ by set inclusion. Given a chain $T$ in $C$, the union of the sets in $T$ is also in $C$: Given $n \in \N$ and distinct $A, B_1, \ldots, B_n \in \bigcup T$, there must be some set in $T$ containing the sets $A, B_1,\ldots,B_n$, hence these sets are nontrivial, finite dimensional, $\Gamma$-invariant, and $A \cap (B_1 + \cdots + B_n) = \{0\}$, which means that $\bigcup T \in C$. Therefore, any chain in $C$ has an upper bound, and by Zorn's lemma $C$ contains a maximal element $M$.
	
Let $P$ be the sum of all sets in $M$ and let $\overline P$ be the closure of $P$ in the uniform topology of $\Ccal(X; \C)$. If $\overline P$ is not equal to $\Ccal(X; \C)$, then by Lemma~\ref{lem:peter-weyl core} there exists a finite dimensional, $\Gamma$-invariant subspace $V$ of $\Ccal(X; \C)$ containing a vector $u$ that does not lie in $P$. The cyclic subspace $W = \mathrm{span}\{L(\gamma) u : \gamma \in \Gamma\}$ has trivial intersection with $P$ because $L(\gamma) u \not \in L(\gamma) P = P$ for all $\gamma \in \Gamma$. This means that $M \cup \{W\}$ is a linearly independent set of subspaces. The space $W$ is $\Gamma$-invariant, and moreover, $W$ is finite dimensional since it is a subspace of the finite dimensional space $V$.  So $M \cup \{W\}$ is contained in $C$. This contradicts maximality of $M$, so $\overline P$ must be equal to $\Ccal(X; \C)$.

Since the representation in $M$ are finite dimensional, by Maschke's theorem for compact groups they decompose into irreducible subrepresentations of $\Ccal(X; \C)$, and we may assume $M$ to be a linearly independent set of $\Gamma$-irreducible subspaces of $\Ccal(X; \C)$ whose sum is uniformly dense.
	
Denote by $m_\pi \in \{0,1,\ldots,\infty\}$ the number of representations in $M$ that are equivalent to $\pi$. Select appropriate orthonormal bases $f_{\pi,i,1},\ldots,f_{\pi,i,d_\pi}$ of the representations in $M$, so that the span of
\[
\big\{f_{\pi,i,j} : \pi \in \hat \Gamma,\, i \in [m_\pi],\, j \in [d_\pi]\big\}
\]
is uniformly dense in $\Ccal(X; \C)$, and so that there are $\Gamma$-equivariant unitary operators $T_{\pi,i,i'} \colon \Hcal_{\pi, i} \to \Hcal_{\pi,i'}$ with $\smash{f_{\pi, i', j} = T_{\pi,i,i'} f_{\pi,i,j}}$ for all $\pi$, $i$, $i'$, and $j$. Now give this system any ordering where $f_{\pi,i,j}$ occurs before $f_{\pi,i',j'}$ whenever $i < i'$, and apply the Gram--Schmidt process to obtain a complete orthonormal system $\{e_{\pi,i,j}\}$. By Lemma~\ref{lem:invariant inner product} we have
\[
e_{\pi,i,j} = f_{\pi,i,j} - \sum_{i'=1}^{i-1} \langle e_{\pi,i,j}, e_{\pi,i',j} \rangle e_{\pi,i',j}  = f_{\pi,i,j} - \sum_{i'=1}^{i-1} c_{\pi,i,i'} e_{\pi,i',j},
\]
where $c_{\pi,i,i'} = \langle e_{\pi,i,j}, e_{\pi,i',j} \rangle$ does not depend on $j$. It follows that the system $\{e_{\pi,i,j}\}$ is symmetry adapted, which completes the proof.
\end{proof}

We now use the previous theorem to prove that the union of the sequence of inner approximations constructed in Section~\ref{subsecSymmetry adapted systems and zonal matrices} is uniformly dense. For this we need one more lemma, for which a proof can be found in \cite{Laat2016}, which is a generalization of a proof from \cite{DeCorte2015}.
\begin{lemma}
\label{lemma:positve definite}
A $\Gamma$-invariant kernel $K \in \Ccal(X \times X)$ is positive definite if and only if $\hat K(\pi)$ is positive semidefinite for every $\pi \in \hat\Gamma$.
\end{lemma}

Now we can show the sequence of inner approximations converges. A similar results is shown in \cite{Bachoc2010}, but there it is required that $\Gamma$  is contained in a bigger group that has a transitive action. Using the existence of a symmetry adapted system as proved above we can avoid this requirement. 

\begin{theorem}
\label{thm:energy:uniformly dense union}
The cone $\bigcup_{d=0}^\infty C_d$ is uniformly dense in $\Ccal(X \times X; \C)_{\succeq 0}^\Gamma$.
\end{theorem}
\begin{proof}
	By Theorem~\ref{thm:symmetry adapted system} there exists a symmetry adapted system $\{e_{\pi,i,j}\}$ of $X$. Define the sets
	\[
	S_d = \mathrm{span} \Big\{e_{\pi,i,j} : \pi \in \hat\Gamma, \, i \in R_{\pi,d},\, j \in [d_\pi] \Big\},
	\]
	so that we have the inclusions
	\[
	S_0 \subseteq S_1 \subseteq \ldots \subseteq \Ccal(X; \C),
	\]
	and $\bigcup_{d=0}^\infty S_d$ is uniformly dense in $\Ccal(X; \C)$.

	Lemma~\ref{lem:approximate identity} shows there exists an approximate identity $\{I_n\}$ of $X$, where each $I_n$ is real-valued, symmetric, and $\Gamma$-invariant. By the above we have that for each $n$, there exists a sequence $\{I_{n,d}\}_d$ of real-valued kernels, with $I_{n,d} \in S_d \times S_d$ for each $d$, such that  $I_{n,d} \to I_n$ uniformly as $d \to \infty$.
	
	We may assume the kernels $I_{n,d}$ to be symmetric and $\Gamma$-invariant: We have that $\bar I_{n,d}$ converges to $I_n$ uniformly, where $\bar I_{n,d}$ is the kernel defined by integrating against the Haar measure of $\Gamma$:
	\[
	\bar I_{n,d}(x,y) = \frac{1}{2} \int \Big(I_{n,d}(\gamma x, \gamma y) + I_{n,d}(\gamma y, \gamma x)\Big) \, d \gamma.
	\]
	And since $S_d$ is $\Gamma$-invariant, we have $\bar I_{n,d} \in S_d \otimes S_d$ for all $n$ and $d$, so we can replace $I_{n,d}$ by $\bar I_{n,d}$.
	
	For each $n$, let $d_n$ be an integer such that $\|I_{n,d_n} - I_n \|_\infty \leq 1/n$. It follows that $\{I_{n,d_n}\}_n$ is an approximate identity of $X$ with $I_{n,d_n} \in S_{d_n}^2$ for all $n$.
	
	For each $n$ we define the kernel $\Pi_n \in \Ccal(X^2 \times X^2)$ by
	\[
	\Pi_n((x,y),(x',y')) = I_{n,d_n}(x,x') I_{n,d_n}(y,y') \quad \text{for} \quad x,x',y,y' \in X.
	\]
	In the remainder of the proof we show $\{\Pi_n\}$ is an approximate identity of $X^2$, and we show that the range of $T_{\Pi_n}$ is contained in $C_{d_n} \subseteq \bigcup_{d=0}^\infty C_d$.
	
	For $f,g \in \Ccal(X; \C)$, we have
	\begin{align*}
	\|T_{K_n} (f \otimes g) - f \otimes g\|_\infty 
	&= \|T_{I_{n,d}} f \otimes T_{I_{n,d}} g - f \otimes g\|_\infty\\
	&\leq \|T_{I_{n,d}} f \|_\infty \|T_{I_{n,d}}g-g\|_\infty + \|T_{J_n} f -f\|_\infty \|g\|_\infty \to 0.
	\end{align*}
	The span of kernels of the form $f \otimes g$ is uniformly dense in $\Ccal(X \times X; \C)$, hence, given a kernel $K \in \Ccal(X \times X; \C)_{\succeq 0}^\Gamma$, the sequence $T_{\Pi_n} K$ converges uniformly to $K$. This shows $\{\Pi_n\}$ is an approximate identity.
	
	The kernel $T_{\Pi_n} K$ lies in $S_{d_n} \otimes S_{d_n}$ because $I_{n,d_n}$ lies in $S_{d_n} \otimes S_{d_n}$. Thus,
	\[
	T_{\Pi_n} K = \sum_{\pi,\pi' \in \smash{\hat \Gamma}} \; \sum_{i \in R_{\pi,d_n}} \sum_{i' \in R_{\pi',d_n}} \; \sum_{j = 1}^{d_\pi} \sum_{j' = 1}^{d_{\pi'}} \Big\langle e_{\pi,i,j}, e_{\pi',i',j'} \Big\rangle_{T_{\Pi_n}K} e_{\pi,i,j} \otimes \overline{e_{\pi',i',j'}},
	\]
	where 
	\[
	\Big\langle e_{\pi,i,j}, e_{\pi',i',j'} \Big\rangle_{T_{\Pi_d} K} = \iint T_{\Pi_d}K(x,y) e_{\pi,i,j}(x) \overline{e_{\pi',i',j'}(y)} \, d\mu(x) d\mu(y)
	\]
	is a sesquilinear form whose $\Gamma$-invariance follows from $T_{\Pi_d}K$ being $\Gamma$-invariant. Lemma~\ref{lem:invariant inner product} shows $\langle e_{\pi,i,j}, e_{\pi',i',j'}\rangle_K = 0$ when $\pi \neq \pi'$ or $j \neq j'$, and $\langle e_{\pi,i,j}, e_{\pi,i',j}\rangle_K$ does not depend on $j$. This shows 
	\[
	T_{\Pi_n} K(x,y) = \sum_{\pi \in \widehat\Gamma} \Big\langle \widehat{T_{\Pi_n} K}(\pi), Z_{\pi,d_n}(x,y) \Big\rangle.
	\]
	The kernel $T_{\Pi_n} K$ is positive definite, so by Lemma~\ref{lemma:positve definite}, the matrices $\widehat{T_{\Pi_n} K}(\pi)$ are positive semidefinite. So, the kernel
	\[
	(x,y) \mapsto \langle \widehat{T_{\Pi_n} K}(\pi), Z_{\pi,d_n}(x,y)\rangle
	\]
	lies in $C_{\pi,d_n}$, and $T_{\Pi_n} K \in \sum_{\pi \in \hat\Gamma} C_{\pi,d_n} = C_{d_n}$.
\end{proof}

\section*{Acknoledgements}

We would like to thank Henry Cohn, Dion Gijswijt, and Frank Vallentin for helpful discussions. We also thank the referee whose suggestions helped to improve the
paper.

\end{document}